\documentclass{amsart}[12pt] \usepackage{graphicx}
\usepackage{subfigure}

 \usepackage{stmaryrd}

\usepackage{color}
    \def\a{\alpha}   \def\({\left (} \def\){\right )}
\def\<{\left\langle} \def\>{\right\rangle}

\newtheorem{acknowledgement}{Acknowledgement}

 \newcommand{\abs}[1]{\left\vert#1\right\vert}  
\newcommand{\R}{\mathbb R}

\newtheorem{Theorem1}{Theorem} \newtheorem{Theorem2}[Theorem1]{Theorem} \newtheorem{Theorem3}[Theorem1]{Theorem}

\newtheorem{Lemma1}{Lemma}[section]  \newtheorem{Lemma3}[Lemma1]{Lemma} \newtheorem{Lemma4}[Lemma1]{Lemma} \newtheorem{Lemma5}[Lemma1]{Lemma}
\newtheorem{Lemma6}[Lemma1]{Lemma} \newtheorem{Lemma7}[Lemma1]{Lemma}  \newtheorem{Corollary}[Lemma1]{Corollary}
\newtheorem{Proposition}[Lemma1]{Proposition}

\begin{document}

\title{The energy identity for a sequence of Yang-Mills $\alpha$-connections}

\author {Min-Chun Hong and Lorenz Schabrun}

\address{Min-Chun Hong, Department of Mathematics, The University of Queensland\\ Brisbane, QLD 4072, Australia}  \email{hong@maths.uq.edu.au} \address{Lorenz Schabrun, Beijing International Center for Mathematics Research (BICMR), Peking University\\Beijing, China}  \email{}

\begin{abstract}
 We prove that the Yang-Mills $\alpha$-functional satisfies the Palais-Smale condition, implying the existence of critical points, which are called Yang-Mills $\alpha$-connections.  It was shown in \cite{HTY} that as $\alpha \to 1$,   a sequence  of Yang-Mills $\alpha$-connections converges to a Yang-Mills connection away from finitely many points. We prove an energy identity for such a sequence of Yang-Mills $\alpha$-connections. As an application, we also prove an energy identity for the Yang-Mills flow at the maximal existence time.
\end{abstract}


\maketitle

\section{Introduction}

Let $M$ be a compact, connected, oriented four-dimensional Riemannian manifold and $E$ a vector bundle over $M$ with structure group $G$, where $G$ is a compact Lie group compatible with the natural inner product on
$E$.
  For each connection $A$, the Yang-Mills functional is defined by
 \begin{equation*}\mbox{YM}(A) = \int_M |F_A|^2 \,dV, \end{equation*}
where  $F_A$ is the curvature of $A$. A connection $A$ is a {\it Yang-Mills connection} if  it satisfies the Yang-Mills equation
 \begin{equation}\label{1.1}D_A^*F_A =0\,.  \end{equation}

 There have been many significant results on the existence of smooth Yang-Mills connections on 4-manifolds. In a pioneering work,
 Atiyah,  Hitchin, Drinfel'd   and   Manin \cite {AHDM} constructed   self-dual Yang-Mills connections on $S^4$.     Taubes \cite
{T1} established the existence of self-dual (anti-self-dual) connections over compact $4$-manifolds with positive (negative) definite intersection form.  In \cite{AB}, Atiyah and Bott established the fundamental result of Morse theory for the Yang-Mills equations on Riemann surfaces. Since the Yang-Mills functional does not satisfy the Palais-Smale
condition in dimension four, Taubes \cite {T2} suggested a new framework for Morse theory for the Yang-Mills functional. However, it seems very challenging to find many applications of the new theory. Using $m$-equivariant connections, Sibner, Sibner and Uhlenbeck \cite{SSU} proved the existence of non-self-dual Yang-Mills connections on the trivial $SU(2)$ bundle over $S^4$. There are further results (\cite {Parker1}, \cite {Gri}) about Morse theory for  the equivariant Yang-Mills functional on $4$-manifolds.

On the other hand,  the theory of Yang-Mills connections in dimension four has many similarities with the theory of harmonic maps in dimension two. For maps $u: \tilde M \to N$ from a two-dimensional manifold $\tilde M$ to another manifold $N$, the energy functional $E(u)=\int|\nabla u|^2$  also fails to satisfy the Palais-Smale condition. For this reason, Sacks and Uhlenbeck \cite {SacksUhlenbeck} introduced an $\alpha$-functional $E_\alpha=\int (1+|\nabla u|^2)^{\alpha}$, where $\alpha >1$, to prove the existence of harmonic maps on surfaces. This new approach via the Sacks-Uhlenbeck functional $E_\alpha$ is very powerful and has produced tremendous applications in harmonic maps and related topics. Specifically, the Sack-Uhlenbeck approach is described as follows. Let $u_{\alpha}: \tilde M \to N$ be a sequence of maps, where $u_\alpha$ is a critical point of the functional $E_\alpha$ and $\alpha
>1$. Sacks and Uhlenbeck proved that for a sequence with $\alpha\to 1$ and having uniformly bounded energies, the sequence $\{u_{\alpha}\}$
converges smoothly to a harmonic map $u_{\infty}$  away from at most finitely many points. Moreover, a sophisticated blow-up phenomenon occurs around singularities.

It is natural to ask whether energy is conserved in the blow-up process. Parker \cite{Parker} established  the energy identity for a sequence of smooth harmonic maps, and Ding and Tian \cite {DingTian} proved an energy identity for a sequence of smooth approximate harmonic maps.  It was a long-standing open question to establish the energy identity for the above sequence of critical points $\{u_{\alpha}\}$ of the Sacks-Uhlenbeck functional $E_{\alpha}$. That is, given that the energies are uniformly bounded as $\alpha\to 1$, do there exist harmonic maps $\omega_i: S^2\to N$ with $i=1,\cdots,l$   such that
  \begin{equation}\label{eqn:identity}
      \lim_{\alpha \to 1} E(u_{\alpha})=E(u_{\infty}) +\sum_{i=1}^l E(\omega_i)?
  \end{equation}
For this question,  Li and Wang \cite{LiWang} established the weak identity for the limit $\lim_{\alpha \to 1} E_{\alpha}(u_{\alpha})$. Under an additional assumption,  Lamm  \cite{Lammidentity} and Li and Wang \cite{LiWang} proved that the  energy identity (\ref {eqn:identity}) holds. However, in a recent preprint \cite{LiW}, Li and Wang constructed a sequence of critical points $u_{\a}$ which provides a counterexample and shows that (\ref{eqn:identity}) is not true.

Following the above Sacks-Uhlenbeck approach,  Hong, Tian and Yin \cite{HTY} introduced the Yang-Mills $\alpha$-functional   to establish the existence of Yang-Mills connections. For
$\alpha > 1$, the Yang-Mills $\alpha$-functional $YM_{\alpha}$ is defined by \begin{equation} \label{functional} YM_{\alpha}(A) =\int_M {(1+|F_A|^2)^{\alpha}dV}. \end{equation}

In several respects, the Yang-Mills $\alpha$-functional has nicer properties than the Yang-Mills functional.
   It was shown in \cite{HTY} that the Yang-Mills $\alpha$-flow associated to (\ref{functional}) admits a global smooth solution, whose limit set contains a smooth critical point
   of the Yang-Mills $\alpha$-functional. By considering the limit $\alpha \to 1$, the authors were then able to obtain existence results for Yang-Mills connections and the Yang-Mills flow.
   Furthermore,  Sedlacek \cite{Sedlacek.direct} applied the weak compactness result of Uhlenbeck \cite
{U2} to prove that a minimizing sequence $A_{i}$ of the Yang-Mills functional converges weakly in $W^{1,2}(M\backslash \{x_1, ..., x_l\})$ to a limiting connection $A_{\infty}$, which can be extended to a Yang-Mills connection  in a (possibly) new bundle $E'$  over $M$.  Hong, Tian and Yin \cite {HTY} used the Yang-Mills $\alpha$-flow to find a modified minimizing sequence, which converges
to the same limit in the smooth topology up to gauge transformations away from finitely many singular points, which  improved the result in \cite{Sedlacek.direct}.  Moreover, for the modified
minimizing sequence, an energy identity was proved.

In this paper, we continue the program of the Yang-Mills $\alpha$-functional and study the critical points of the Yang-Mills $\alpha$-functional directly. First, we show

\begin{Theorem1} \label {Theorem1} The functional $YM_{\alpha}$ satisfies the Palais-Smale condition for every $\alpha >1$. \end{Theorem1}

The  Palais-Smale condition (e.g.
\cite{Palais.Foundations}, \cite  {Struwe}, \cite {Lammidentity}) guarantees the existence of a critical point $A_\alpha$ of $YM_\alpha$; i.e. $A_{\alpha}$ satisfies \begin{equation}
\label{eulerlagrange1} D^*_A \left( (1+|F_A|^2)^{\alpha-1}F_A \right) = 0 \end{equation} in the weak sense. In fact, the smoothness of the Yang-Mills $\alpha$-connection $A_\alpha$ is essentially due to Isobe in \cite{Isobe}, and we prove it in the Appendix. We call $A_{\alpha}$ a Yang-Mills $\alpha$-connection if it satisfies (\ref{eulerlagrange1}).

In analogy with the above Sacks-Uhlenbeck program, we consider attempting to construct a Yang-Mills connection as the limit as $\alpha \to 1$ of a sequence $A_\alpha$ of critical points
of $YM_\alpha$. In fact, it was proved in \cite{HTY}  that for a sequence $\alpha_k\to 1$, the sequence $\{A_{\alpha_k}\}$ sub-converges smoothly, up to gauge transformations, to a
Yang-Mills connection $A_{\infty}$ away from at most finitely many points. A natural question to ask is whether energy is lost during the blowup process. Fortunately, we are able to prove the
following energy identity:

\begin{Theorem2} \label{maintheorem} Let  $A_\alpha$ be a sequence of Yang-Mills $\alpha$-connections on $E$ with $\alpha\to 1$, and $YM(A_{\alpha})$ uniformly bounded. Then there exists a finite set $S\subset M$ and a sequence of gauge transformations $\phi_{\alpha}$ defined on $M\setminus S$, such that for any compact $K\subset M\setminus S$, a subsequence of $\phi_\alpha^*A_{\alpha}$ converges to $A_\infty$ smoothly in $K$. Moreover, there are a finite number of bubble bundles $E_1,\cdots,E_l$ over $S^4$ and Yang-Mills connections $\tilde{A}_{1,\infty},\cdots,\tilde{A}_{l,\infty}$ such that \begin{equation} \label{energyidentity} \lim_{\alpha \to 1} YM(A_{\alpha}) = YM(A_\infty)  +
\sum_{i=1}^l YM(\tilde A_{i,\infty}). \end{equation}

\end{Theorem2} This result is surprising in comparison with those of Li-Wang \cite{LiWang}, \cite{LiW} and Lamm \cite{Lammidentity}, since we are able to prove the full energy identity for Yang-Mills $\alpha$-connections.
In fact,  Li-Wang \cite{LiWang} and   Lamm \cite{Lammidentity} used the  idea of  Sacks-Uhlenbeck \cite {SacksUhlenbeck} on removable singularities of harmonic maps in 2 dimensions to obtain
that \[\lim_{\delta\to 0}\lim_{R\to\infty}\lim_{\a\to 1}\int_{B_{\delta}\backslash B_{Rr_{\a} }(x_{\a})}|u_{\a, \theta}|^2\,dx =0,\] where $x_{\alpha}$ tends to a singularity $x_i$.

 But they had to handle the term
\[\lim_{\delta\to 0}\lim_{R\to\infty}\lim_{\a\to 1} \int_{B_{\delta}\backslash B_{Rr_{\a} }(x_{\a})}|\frac {\partial u_{\a}}{\partial r} |^2d x\] by using a type of Pohozaev's identity which involves a very `bad' term, so they could not prove the energy identity. In contrast, the idea of Uhlenbeck \cite{Uhl.removable} on removable singularities of Yang-Mills connections in 4 dimensions is to
construct a broken Hodge gauge, which is different from one in \cite {SacksUhlenbeck}.   We can apply the idea of the broken Hodge gauge  to show that \[ \lim_{R \to \infty}
\lim_{\delta \to 0} \lim_{\alpha \to 1} \int_{B_\delta \setminus B_{Rr_\alpha}(x_{\a})}{|F_{A_\alpha}|^2 dV} = 0. \] See Lemma 3.7.

Tian \cite{Tian.calibrated} first established the energy identity of a sequence of Yang-Mills $\Omega$-self-dual connections. Rivi\`{e}re \cite{Riviere} proved the energy identity for a
sequence of smooth Yang-Mills connections in 4 dimensions.  As a consequence of Theorem \ref{maintheorem}, we can give a different proof of this result (see Corollary \ref{corol}). As another application, we are also able to give a new proof of the energy identity results of \cite {HTY} (see Proposition \ref{prop}).

It is also possible to consider whether an energy identity holds at the blowup time for the corresponding heat flows. Struwe \cite {St1} established the global existence of the harmonic map flow in 2 dimensions.  Ding-Tian \cite {DingTian} (see also \cite{LW}) established the energy identity for the harmonic map flow at the  maximal existence time.  Motivated these results,  we can  prove a similar result for  the Yang-Mills flow in 4 dimensions.  In fact, Struwe \cite{StruweYMheat} established the global existence of weak solutions of the Yang-Mills flow
\begin{equation}\label{YMHF}\frac {\partial A}{\partial
t}=-D_A^*F_A\quad  \mbox{in  }M\times [0, \infty)\end{equation}
 with initial value $A(0)=A_0$.
Struwe proved that   the weak solution $A(t)$ of the Yang-Mills flow  is, up-to gauge transformations, smooth in $M\times (0,T)$ for some maximal time $T>0$. Moreover, as $t \to T$ the solution $A(t)$ converges, up-to gauge transformations, to a connection $A(T)$, smoothly away from at most finitely many points. Schlatter \cite{AS}  gave the details for the blow-up analysis of  the Yang-Mills flow at the singular time $T$, but there is no result concerning the energy identity of  the Yang-Mills flow at the time $T$. We can  apply  the proof of Theorem \ref{maintheorem} to establish an  energy identity for the Yang-Mills flow as follows:

\begin{Theorem3} \label{thm3}Let $A(t)$ be a solution to the Yang-Mills flow (\ref {YMHF}) in $M\times[0,T)$, where $T\in (0, \infty ]$ is the maximal existence time, and $A(t)$ converges
weakly as $t \to T$ to a connection $A(T)$. Then there are a finite number of bubble bundles $E_1,\cdots,E_l$ over $S^4$ and Yang-Mills connections
$\tilde{A}_{1,\infty},\cdots,\tilde{A}_{l,\infty}$ such that \begin{equation*} \lim_{t \to T} YM(A(t)) = YM(A(T)) + \sum_{i=1}^l YM(\tilde{A}_{i,\infty}). \end{equation*} \end{Theorem3} In
particular, if $T=\infty$,  Theorem \ref{thm3} implies  a new proof of an identity for a sequence of the second Chern numbers  of holomorphic vector bundles over K\"ahler surfaces (Theorem
11 of \cite {HT}, Theorem 4 of \cite{DW}).

The paper is organised as follows. In section 2, we recall some necessary background and prove Theorem 1. In section 3, we prove the energy identities Theorem 2 and Theorem 3. Finally, in the Appendix, we show the regularity of weak solutions to the Euler-Lagrange equations.

\section{Preliminaries and the Palais-Smale condition}

We begin by recalling background and introducing the notation we will require. Denote by $\mathcal{A}$ the affine space of metric connections on $E$. Let $A \in \mathcal{A}$ be a
metric connection. After fixing a reference connection $D_0 =D_{A_0}$, we write $D_A = D_0 + A$ where $A \in \Gamma(\operatorname{Ad} E \otimes T^*M)$. The curvature of $A$ is $F_A
\in \Gamma(\operatorname{Ad} E \otimes \Lambda^2 T^*M)$. Here we denote by $\operatorname{Ad} E$ the adjoint bundle which has fibre $\mathfrak{g}$, the Lie algebra of $G$.

The inner products on a pair of vector bundles induce an inner product on their tensor product. In the case of $\operatorname{Ad} E \subset \operatorname{End} E = E \otimes E^*$, this
is just the negative of the Killing form. This combines with the inner product on $T^*M$ to define an inner product on $\operatorname{Ad} E \otimes^k T^*M$. Similarly, the connections
on a pair of vector bundles induce a connection on their tensor product. Thus we obtain a connection on $\operatorname{Ad} E \subset \operatorname{End} E$ and, using the Levi-Civita
connection, a connection on $\operatorname{Ad} E \otimes^k T^*M$. The Levi-Civita connection also allows us to define multiple iterations of the covariant derivative on any vector
bundle.

We denote by $D_A$ both the exterior covariant derivative on $\Gamma(E \otimes \Omega^p(M))$, and the exterior covariant derivative on $\Gamma(\operatorname{End} E \otimes
\Omega^p(M))$. The functional $YM_{\alpha}$ has a Fr\'{e}chet derivative
\begin{align} \label{derivexpress} dYM_{\alpha, A}(B) & = \left. \frac{d}{d t} \right|_{t  = 0} YM_{\alpha}(A + tB) =
\int_M {2 \alpha \left( 1 + |F_A|^2 \right)^{\alpha - 1}\left\langle F_A,D_A B\right\rangle dV}
 \\ & = 2 \alpha  \int_M \left\langle D_A^*(1+|F_A|^2)^{\alpha -1}F_A),  B\right\rangle dV,  \nonumber\end{align}
 so
 the Euler-Lagrange equation for $YM_{\alpha}$ is
(\ref {eulerlagrange1}). Note that solutions to (\ref {eulerlagrange1}) are smooth up to gauge by the results of the appendix. Then since $(1+|F_A|^2)^{\alpha-1}$ is bounded below,
this equation is equivalent to \begin{equation} \label{eulerlagrange2} D^*_AF_A - (\alpha-1)\frac{*(d|F_A|^2\wedge*F_A)}{1+|F_A|^2}=0. \end{equation}

The functional $YM_{\alpha}$ is invariant under the action of the gauge group $\mathcal  G$, which consists of those automorphisms of $E$ which preserve the inner product. A gauge
transformation is thus a section of the bundle $\operatorname{Aut} E$, which has fibre $G$. The gauge group acts on $\mathcal{A}$ by \[ s^*D_A = s^{-1} \circ D_A \circ s. \] Locally, we write $D_A=d+A$.
$s^*D_A = d + s^*A$, we find \begin{equation} \label{gaugetransform} s^*A = s^{-1}d s + s^{-1} A s. \end{equation}
The functional $YM_{\alpha}$  is gauge invariant, so it is more appropriately viewed as a functional \[ YM_{\alpha}: \mathcal{A} / \mathcal G \to \mathbb{R}, \] where $\mathcal{A} / \mathcal G$ is a space with gauge equivalence and
imbued with the quotient topology, but not a smooth manifold.

We first must topologise $\mathcal{A}$ by a suitable choice of norm. Given a section $u$ of a vector bundle $V$ with a smooth connection $D_0=D_{A_0}$, we define the $W^{k,p}$
Sobolev norm by \[ \|u\|_{W^{k,p}} = \left( \displaystyle\sum_{|\alpha| =0}^{k} \int_M{|\nabla_{A_0}^{\alpha}u|^p}\right)^{1/p}. \] Different choices of reference connection lead to
equivalent norms. We then define the Sobolev space $W^{k,p}(V)$ to be the completion of the smooth sections of $V$ by the $W^{k,p}$ norm.

We may define Sobolev norms on the space of connections using the vector bundle $\operatorname{Ad} E \otimes T^*M$. A different choice of reference connection taken as the origin of
the affine space $\mathcal{A}$ also leads to an equivalent norm. We denote by $\mathcal{A}^{1,p} = W^{1,p}(\mathcal{A})$ the space of connections of class $W^{1,p}$.

Since a gauge transformation $s$ can be expressed locally over $U \subset M$ as a map $s:U \to G$, we will say that a gauge transformation is of class $W^{k,p}$
if its local restrictions are of class $W^{k,p}$, and $s(x) \in G$ almost everywhere. Due to the derivative of the gauge transformation appearing in (\ref{gaugetransform}), we need to
control two derivatives of the gauge transformation, so we consider $\mathcal G^{2,p}$. We will choose $p = 2 \alpha$, and consider the space $\mathcal{A}^{1,2 \alpha}$ and its
quotient $\mathcal{A}^{1,2 \alpha} / \mathcal G^{2,2 \alpha}$.

Now, we prove that the Yang-Mills $\alpha$-functional satisfies the Palais-Smale condition on the quotient space.

 \begin {proof} [Proof of Theorem 1] We say that a sequence $A_n \in \mathcal{A}^{1,2 \alpha}/ \mathcal G^{2,2 \alpha}$ is a Palais-Smale sequence if it satisfies
\begin{align*} \text{i) } & YM_{\alpha}(A_n) \leq C, \\ \text{ii) } & \|d YM_{\alpha, A_n}\|_{W_{A_n}^{-1, q}} \to 0, \end{align*}
where $q$ is the dual number of $p=2\alpha$; i.e. $\frac 1 p+\frac 1 q=1$.
Similarly to one in \cite {SSU}, the norm $\|\cdot\|_{W_{A_n}^{-1, q}}$ here is the dual norm  of $W_{A_n}^{1,2 \alpha}$ defined by
\[\|d YM_{\alpha, A}\|_{W_A^{-1, q}}=\frac {\mbox{max} \frac {d YM_{\alpha}(A+\varepsilon \delta A)}{d\varepsilon  }|_{\varepsilon =0}}{\|\delta A\|_{W_A^{1,2\alpha}}},\]
where
\[\|\delta A\|^{2\alpha}_{W_A^{1,2\alpha}}=\int_M |\nabla_A \delta A|^{2\alpha}+|\delta A|^{2\alpha}\,dV.\]
We say that the functional (\ref{functional}) satisfies the Palais-Smale condition if every Palais-Smale sequence contains a subsequence that converges in $W^{1,2\alpha}$ up to gauge transformations.
That is, passing to a subsequence, $s_n^* A_n$ converges for some sequence of gauge transformations $s_n \in \mathcal G^{2,p}$. Note that this condition implies Condition C of
\cite{PalaisMorseHilbert}.

First we observe that the $L^2$ norm of the curvature cannot concentrate, since by H\"older's inequality, \[ \int_{B_r}{|F_{A_n}|^2dV} \leq \left( \int_{B_r}{|F_{A_n}|^{2\alpha}dV} \right)^{\frac{1}{\alpha}}
|B_r|^{\frac{\alpha-1}{\alpha}} \to 0 \] as $r \to 0$. For a sufficiently small ball $U=B_r$,  the bundle $E$ can be trivialized so that we can assume $D_{A_n}=d+A_n$ on $U$. Then
 we apply the gauge-fixing theorem of Uhlenbeck (Theorem 1.3 of \cite{U2}) to obtain that there are  local gauge transformations
$s_n$ defined on $U$ such that
 \[\left. d^*(s_n^*A_n) \right |_U=0,\quad s_n^*A_n\cdot \nu|_{\partial U} =0 \]
 and
 \[ \|s_n^*A_n\|_{W^{1,2 \alpha}(U)}\leq C\|F_{A_n} \|_{L^{2 \alpha}(U)}.\]
 Since this result is local, we must do this in each element $U_j$ of a cover of $M$, and check that in the limit the local connections can be
sewn together to yield a global connection. In particular, one must check that the transition functions between elements of the cover are converging. This procedure is described by
Sedlacek \cite{Sedlacek.direct}.

Since $M$ is compact, there is  a finite open cover $\{U_j\}_{j=1}^L$ of $M$, where $U_j=B_r(x_j)$ for some sufficiently small $r>0$, and at each $x\in M$, at most a finite number of the balls intersect.
From the boundedness of $YM_{\alpha}(A_n)$, we have that $F_{A_n}$ is bounded in $W^{1,2\alpha}$. For simplicity,  we define $\bar A_{n,j} = s_{n}^*A_n$ on each ball $U_j$. Then $d+\bar A_{n,j}$ can be regarded as a local representative of  the  global connection $D_{A_n}$ of $E$ over $U_j$. Moreover, in the overlap $U_i\cap U_j$ of two balls,  $\bar A_{n,i}$ and $\bar A_{n,j}$ can be identified as the same by a gauge transformation $s_{ij}\in \mathcal G^{2,2 \alpha}$ between  $E_{U_i}$ and $E_{U_j}$ (see Lemma 3.5 of \cite{U2}, also \cite{Sedlacek.direct}).  Thus we can glue $\bar A_j$ together to obtain a global connection $D_{\bar A}$, which is $\mathcal G^{2,2 \alpha}$-gauge equivalent to $D_{A_n}$  in the local trivialization of $E_{U_j}$. Passing to a subsequence, it follows
from the Rellich-Kondrachov Theorem that $\bar A_{n,j}$ converges strongly in $L^{2\a}(U_j)$, implying that  $\bar A_{n,j}$ converges  in $C^0$ and are uniformly bounded.

Since $YM_\alpha(\bar A_n; M)$ and $\|\bar A_{n, j} - \bar A_{m, j}\|_{W^{1,2\alpha}(U_j)}$ are bounded, $\|\bar A_{n, j} - \bar A_{m, j}\|_{W_{\bar A_n}^{1,2\alpha}(U_j)}$ is also bounded. Then we also have
\begin{equation} \label{derivativetozero} |dYM_{\alpha,
\bar A_n}(\bar A_n - \bar A_m)|\leq \|d YM_{\alpha,\bar A_n}\|_{W_{\bar A_n}^{-1,q}} \,  \sum_{j=1}^L \|\bar A_{n, j} - \bar A_{m, j}\|_{W_{\bar A_n}^{1,2\alpha}(U_j)}  \to 0. \end{equation}
Using H\"older's inequality and the fact that $\|\bar A_{n, j} - \bar A_{m, j}\|_{L^{4 \alpha}(U_j)} \to 0$, we  find
\begin{align*}  & \int_{U_j}\left( 1 + |F_{\bar A_n}|^2 \right)^{\alpha - 1}|F_{\bar A_n}| \,|\bar A_{n,j}|\, |\bar A_{n,j} - \bar A_{m,j}|\,
dV \\ \leq  & \left (\int_{U_j} \left( 1 + |F_{\bar A_n}|^2 \right)^{\alpha }dV\right )^{\frac{2\alpha-1}{2\alpha}}
\left (\int_{U_j}\,|\bar A_{n, j}|^{2\alpha}\, |\bar A_{n, j} - \bar A_{n, j}|^{2\alpha}\,dV\right )^{\frac 1 {2\alpha}}
\\ &\to 0, \quad \mbox {as }m, n\to \infty. \end{align*}
Denote by $o(1)$ terms which are going to zero as $n,m \to \infty$. Then
\begin{align*}
dYM_{\alpha,  \bar A_n}(\bar A_n - \bar A_m)& = \int_M  2 \alpha\left <D^*_{\bar A_n}[\left( 1 + |F_{\bar A_n}|^2 \right)^{\alpha - 1} F_{\bar A_n}], (\bar A_n - \bar A_m)\right >
dV \\ = & \int_M 2 \alpha \left( 1 + |F_{\bar A_n}|^2 \right)^{\alpha - 1}\left< F_{\bar A_n},D_{\bar A_n} (\bar A_n - \bar A_m)\right >
dV \\ = & \int_M 2 \alpha \left( 1 + |F_{\bar A_n}|^2 \right)^{\alpha - 1}\left\langle F_{\bar A_n},D_{A_0}(\bar A_n - \bar A_m)\right\rangle dV\\
& + \int_M 2 \alpha \left( 1 + |F_{\bar A_n}|^2 \right)^{\alpha - 1}\left\langle F_{\bar A_n},\bar A_n\#(\bar A_n - \bar A_m)\right\rangle dV \\ =
& \int_M 2 \alpha \left( 1 +
|F_{\bar A_n}|^2 \right)^{\alpha - 1}\left\langle F_{\bar A_n},F_{\bar A_n}-F_{\bar A_m})\right\rangle dV +o(1), \end{align*} since $A \wedge A - B \wedge B = A \wedge(A-B) - (B - A)
\wedge B$.

  It is well-known that  there is a constant $c>0$  such that
\[ \left\langle ( 1 + |b|^2)^{\alpha - 1}b-( 1 + |a|^2)^{\alpha - 1}a, b-a \right\rangle \geq c|b-a|^{\alpha} \]  for any two constants $a,b \in \mathbb{R}^k$. Using this inequality, we obtain
\begin{align*} & (dYM_{\alpha,\bar A_n} - dYM_{\alpha, \bar A_m})(\bar A_n - \bar A_m) \\ = & \int_M {2 \alpha \left\langle  \left( 1 + |F_{\bar A_n}|^2 \right)^{\alpha - 1} F_{\bar A_n}
- \left( 1 + |F_{\bar A_m}|^2 \right)^{\alpha - 1} F_{\bar A_m},F_{\bar A_n}-F_{\bar A_m})\right\rangle dV} +o(1) \\ \geq & 2\alpha c\int_M{|F_{\bar A_n}-F_{\bar A_m}|^{2 \alpha}dV} + o(1) \\
\geq & 2\alpha c\int_{U_j}{|F_{\bar A_{n,j}}-F_{\bar A_{m,j}}|^{2 \alpha}dV} + o(1). \end{align*}
Note that $F_{\bar A_{n,j}}=d\bar A_{n,j}+[\bar A_{n,j},\bar A_{n,j}]$ and  $F_{\bar A_{m,j}}=d\bar A_{m,j}+[\bar A_{m,j},\bar A_{m,j}]$ in $U_j$. Then
\[\int_{U_j} |d(\bar A_{n,j} - \bar A_{m,j})|^{2\alpha}\,dV=o(1).\]
 Since $d^*(\bar A_{n,j} - \bar A_{m,j})=0$    in
$U_j$ and $(\bar A_{n,j} - \bar A_{m,j})\cdot\nu =0$ on  $\partial U_j$,    it follows from Lemma 2.5 of \cite {U2} that
 \[\int_{U_j} |\nabla (\bar A_{n,j} - \bar A_{m,j})|^{2 \alpha}\,dV\leq C \int_{U_j} |d(\bar A_{n,j} - \bar A_{m,j})|^{2\alpha}\,dV\to 0. \] Thus $\bar A_{n,j}=s_n^* A_n$ is converging strongly in $W^{1,2\alpha}(U_j)$ as required. After sewing
together the patches $U_j$ in the covering, we get a globally defined connection $A_\infty$. For any smooth $B$ with support in $U_i$, we have
\begin{align}  dYM_{\alpha, A_n}(B)  =
\int_M {2 \alpha \left( 1 + |F_{A_n}|^2 \right)^{\alpha - 1}\left\langle F_{A_n},D_A B\right\rangle dV}
\end{align}  Since  $\|d
YM_{\alpha ,A_n}\|_{W_{A_n}^{-1, q}} \to 0$, the strong convergence of $A_{n}$ in $W^{1, 2\a}$ implies that
\begin{align}
\int_M { \left( 1 + |F_{A_{\infty}}|^2 \right)^{\alpha - 1}\left\langle F_{A_{\infty}},D_A B\right\rangle dV}=0.
\end{align} Therefore  $A_{\infty}$ is a critical point of  $YM_\alpha$.
 This completes a proof of Theorem 1. \end{proof}

The Palais-Smale condition implies the existence of a critical point. In particular, given a min-max sequence of $YM_{\alpha}$ (e.g. \cite{Struwe}, \cite{Lammidentity}, \cite{SSU}), the Palais-Smale condition  then ensures, for each
$\alpha > 1$, the existence of a critical point $A_{\alpha}$ of $YM_{\alpha}$. Furthermore, we have the following remark.

\vspace{3mm} \noindent {\bf Remark} \emph{ By using an idea of Struwe \cite{Struwe} (see also Lamm \cite{Lammidentity}), one expects that there exists a sequence of critical points
$A_{\alpha}$ of the Yang-Mills $\alpha$-functional with $\alpha \to 1$ such that} \begin{equation} \liminf_{\alpha \to 1} (\alpha - 1) \int_M
{\log(1+|F_{A_\alpha}|^2)(1+|F_{A_\alpha}|^2)^\alpha dV} =0. \end{equation} However, we do not require this result for this paper, so we will not give a proof here.

\section{Energy Identity} In this section, we establish the energy identity for a sequence of Yang-Mills $\alpha$-connections (Theorem  \ref{maintheorem}), and the energy identity for the Yang-Mills flow at the maximal existence time (Theorem \ref{thm3}).

Consider now a sequence $A_{\alpha}$ of connections, where $A_{\alpha}$ is a smooth critical point of $YM_{\alpha}$, and $\alpha \to 1$. We suppose that they have uniformly bounded energy
$YM(A_{\alpha}) \leq K$ for some constant $K$. We now study the limit of the sequence as $\alpha \to 1$. In fact,  it was shown in  Lemmas 3.5-3.6 of \cite{HTY} that a subsequence of $A_\alpha$,
up to gauge transformations, converges smoothly  to a connection $A_{\infty}$ away from finitely many points $x_i \in M$, $1 \leq i \leq l$. Moreover, there is a constant $\varepsilon_0>0$ such that  the singular points $\{x_i\}$ are defined by the condition
\[ \limsup_{\alpha \to 1} YM(A_\alpha; B_R(x_i)) \geq \varepsilon_0 \] for any $R\in (0, R_0]$, with some fixed $R_0>0$.

We now consider the energy identity (\ref{energyidentity}) for the above sequence of Yang-Mills $\alpha$-connections $A_{\alpha}$. We fix an energy concentration (singular) point $x_i$,
and choose $R_0$ so that $B_{R_0}(x_i)$ contains no concentration points other than $x_i$.

 In order to establish the energy identity, we recall the  removable singularity theorem of Uhlenbeck \cite {U1} and the gap theorem of  Bourguignon and Lawson \cite{BL}: There is a constant $\varepsilon_{g}>0$ such that if $A$ is a Yang-Mills connection on $S^4$ satisfying $\int_{S^4}|F_A|^2<\varepsilon_{g}$, then $A$ is flat; i.e. $F_A=0$ on $S^4$.

We also have:

    \begin{Lemma1} ($\varepsilon$-regularity estimate) \label{epsilonregularity}
        \label{lem:mainlemma}  There exists $\varepsilon_0>0$ such that if $A_\alpha$ is a smooth Yang-Mills $\alpha$-connection satisfying
    \begin{equation*}
     \int_{B_{R}(x_0)} |F_{A_\alpha}|^2 dV< \varepsilon_0,
    \end{equation*}
    then we have
  \begin{equation*}
    \sup_{B_{R/2}(x_0)} |F_{A_\alpha}|\leq \frac{C}{ R^2} \left(\int_{B_R(x_0)} |F_{A_\alpha}|^2 dV\right)^{1/2}.
  \end{equation*}
 \end{Lemma1}
 \begin {proof} Without loss of generality, we can assume $R=1$.
Using Lemma 3.5 of \cite{HTY},   $F_{A_\alpha}$ is bounded in $B_{3/4}(x_0)$ by a constant $C$.
Recalled that each   $\alpha$-connection $A_{\alpha}$ has  the  Bochner tye formula  (see Lemma 3.2 of \cite{HTY}):
For $\alpha-1$  sufficiently small, there is a constant $C$ such that
  \begin{eqnarray} \label{eqn:firstbochner}
  - \nabla_{e_i} \left( (\delta_{ij} +2(\alpha-1)\frac{
   \left <F_{lj},F_{li}\right >}{1+\abs{F}^2}) \nabla_{e_j} \abs{F}^2\right) \leq C\abs{F}^2 (1+\abs{F}),\nonumber
  \end{eqnarray}
  where $F=F_{A_\alpha}$. Then applying a variant of the Moser-Harnack estimate, we have
  \begin{equation*}
   \sup_{B_{1/2}(x_0)} |F_{A_\alpha}|\leq C \left(\int_{B_1(x_0)} |F_{A_\alpha}|^2 dV\right)^{1/2}.
  \end{equation*}
  \end{proof}

\subsection{Bubble-neck decomposition}

It is well-known that the  bubble-neck decomposition holds for a sequence of smooth harmonic maps.
There are  two kind of   methods on constructing the bubble tree and neck decomposition  in harmonic maps, one present by Parker  \cite {Parker} and another one   by Ding-Tian  \cite{DingTian}.
In fact, Parker  \cite {Parker} pointed out that a bubble procedure is also true for Yang-Mills connections (without  details).     Rivi\`{e}re \cite {Riviere} pointed out that using the idea of Ding-Tian \cite{DingTian},  multiple bubbles  can be simplified to a single bubble for a sequence of Yang-Mills connections.

In order to make  proofs completely,
 we here give detailed proof on constructing the  bubble-neck decomposition  for a sequence of Yang-Mills $\alpha$-connections by following the idea of Ding-Tian \cite{DingTian}. Some details are  similar to  those in the Appendix of \cite {LiW}. The bubble tree procedure is divided into three steps.

\medskip\noindent{\bf Step 1.} To find a maximal (top) bubble at the level one  (first re-scaling).

After passing to a subsequence we know  that $A_\alpha \to A_{\infty}$, up-to gauge transformations, smoothly in $B_{R_0}$ away from $x_i$, so we have
\begin{equation}\label{local} YM(A_\alpha; B_{R_0} \setminus B_{\delta}
(x_i)) \to YM(A_\infty; B_{R_0} \setminus B_{\delta} (x_i))
\end{equation} for any $\delta\in (0, R_0)$,  where $A_{\infty}$ is a Yang-Mills connection and the singularity of $A_{\infty}$ can be removed (\cite  {Uhl.removable}).

Since $x_i$ is a concentration point, we find sequences $x^1_\alpha \to x_i$  such that
\[|F_{A_{\a}}(x^1_{\a})|=\max_{B_{R_0}(x_i)} |F_{A_{\a}}|,\quad  r_\alpha^1 =\frac 1 {|F_{A_{\a}}(x^1_{\a})|^{1/2}}\to 0. \]
In  the neighborhood of the singularity $x_i$,  $D_{A_{\alpha}}=d+A_{\alpha}$ with $A_{\alpha}=A_{\alpha, k}(x)dx^k$. Then we define the rescaled connection
\[ D_{\tilde A_{\alpha}}(x)=d
+\tilde A_{\alpha}(x):= d + r^1_{\alpha}A_{\alpha, k}(x^1_{\alpha}+r^1_{\alpha}x)dx^k.\]
The connection $\tilde A_{\alpha}=A_{\alpha}(x^1_{\alpha}+r^1_{\alpha}x)$ satisfies
\begin{equation}\label{alpha-B}
 -D_{\tilde A_{\alpha}}^* F_{\tilde A_{\alpha}}+(\alpha-1)\frac{*(d\abs{F_{\tilde A_{\alpha}}}^2\wedge *F_{\tilde A_{\alpha}})}{(r^1_{\alpha})^4+\abs{F_{\tilde A_{\alpha}}}^2}=0.
\end{equation}
Since $\|F_{\tilde A_{\alpha}}\|_{L^{\infty}}=1$.
Using Lemmas 3.6-3.7 of \cite{HTY} again, $\tilde A_{\alpha}$ sub-converges smoothly, up-to gauge transformations, to $\tilde A_{1,\infty}$ locally in $\R^4$  as $\alpha\to 1$, and $\tilde A_{1,\infty}$ can be extended to a connection on $S^4$ (see \cite {Uhl.removable}) and nontrivial. We call  $\tilde A_{1,\infty}$ to be the first bubble, which satisfies
\begin{equation}\label{First}YM(\tilde A_{1, \infty}; \R^4) =\lim_{R \to \infty}\lim_{\alpha \to 1}YM(\tilde A_\alpha; B_{R}(0))=
\lim_{R \to \infty}\lim_{\alpha \to 1}YM(A_\alpha; B_{Rr^1_\alpha}(x_\alpha)). \end{equation}

\medskip\noindent{\bf Step 2.} To find out new bubbles at the second level (second re-scaling).

 Assume that for a fixed small constant $\varepsilon>0$ (to be chosen later),  there exist two positive constants $\delta$ and   $R$ with $Rr^1_\alpha <4 \delta$ such that for all $\alpha$ sufficiently close to $1$, we have
\begin{equation}\label{small} \int_{B_{2r} \setminus
B_r (x^1_{\a})}{|F_{A_\alpha}|^2 dV} \leq \varepsilon \end{equation}
 for all $r \in (\frac{Rr^1_\alpha}{2},2\delta)$.

If  (\ref{small}) is   true,  it follows from (\ref{local})  and (\ref{First})  that
 \begin{align*}\label{neckenergy}
\lim_{\alpha\to 1} YM(A_{\alpha}; B_{R_0}(x_i))=&YM(A_{\infty}; B_{R_0}(x_i) )+  YM(\tilde A_{1,\infty}; \R^4)
 \\
&+ \lim_{R \to \infty} \lim_{\delta \to 0} \lim_{\alpha \to 1} YM(A_\alpha; B_\delta \setminus B_{Rr^1_\alpha}(x^1_\alpha)). \nonumber
\end{align*}
We can show (see below Lemma \ref{Removable}) that there is no missing energy in the neck region,
i.e. \begin{equation}\label{Key}
 \lim_{R \to \infty} \lim_{\delta \to 0} \lim_{\alpha \to 1} YM(A_\alpha; B_\delta \setminus B_{Rr^1_\alpha}(x^1_\alpha)) =0. \end{equation}
 Therefore, the energy identity follows. In this case, this means that there is only single bubble $\tilde A_{1, \infty}$ around $x_i$.

 If the assumption (\ref{small}) is not true, then
 for any two constants $R$ and  $\delta$ with $  Rr^1_\alpha< 4\delta$, there is a constant $r^2_{\alpha}\in  (\frac{Rr^1_\alpha}{2},2\delta)$  such that
 \begin{equation}\label{Fact}
  \lim_{\alpha \to 1}  \int_{B_{2r^2_{\alpha}}\backslash B_{r^2_{\alpha}}(x^1_{\alpha})} |F_{A_{\alpha}}|^2 >\varepsilon .\end{equation}
 If $
\lim_{\alpha\to 1}\frac {r^2_{\alpha}}{r^1_{\alpha}}\leq C $  for a finite constant $C$, it is not a problem since $\tilde A_{\alpha}$ converges smoothly, up-to gauge transformations, to $\tilde A_{1,\infty}$ locally in $\R^4$.
If $\lim_{\alpha\to 1}r^2_{\alpha}\neq 0$, it can be ruled out by choosing $\delta$ sufficiently small in (\ref{small}). Therefore, we can assume that $\lim_{\alpha\to 1}\frac {r^2_{\alpha}}{r^1_{\alpha}}=\infty $ and $\lim_{\alpha\to 1}r^2_{\alpha}=0$.

Since there might be many different constants $r^2_{\alpha}\in  (\frac{Rr^1_\alpha}{2},2\delta)$ satisfying (\ref{Fact}), we must classify these numbers.
  For any two constants  $r^2_{\alpha}$ and $\tilde r^2_{\alpha}$ in $(\frac{Rr^1_\alpha}{2},2\delta)$ satisfying (\ref{Fact}), they can be classified in different groups by the following properties:
 \begin{align}\label{group1}
 \lim_{\alpha \to 1} \frac {r^2_{\alpha}}{\tilde r^2_{\alpha}}=+\infty &\quad\mbox{or } \quad \lim_{\alpha \to 1} \frac {r^2_{\alpha}}{\tilde r^2_{\alpha}}=0;
  \\ \label{group2}
  0< a\leq \lim_{\alpha \to 1} \frac {r^2_{\alpha}}{\tilde r^2_{\alpha}}<b &\quad\mbox{or } \quad  0< a\leq \lim_{\alpha \to 1} \frac {\tilde r^2_{\alpha}}{r^2_{\alpha}}<b
  \end{align}
  for finite constants $a$ and $b$. We say that $r^2_{\alpha}$ and $\tilde r^2_{\alpha}$ are in the same group if they satisfy (\ref{group2}). Otherwise, they are in different groups if they satisfy (\ref{group1}).

  Since there is a uniformly bounded energy
$YM(A_{\alpha}) \leq K$ for some constant $K$ and $\varepsilon$ is a fixed constant, the above  different groups of $r^2_{\alpha}$ satisfying (\ref{Fact}) must be finite, so that we can choose a smallest  number $r^2_{\alpha}$ of different groups as a re-scalling scale, but $r^2_{\alpha}$ is not in the group of $r^1_{\alpha}$; i.e.
$\lim_{\alpha \to 1} \frac {r^2_{\alpha}}{r^1_{\alpha}}=\infty $ and $\lim_{\alpha\to 1}r^2_{\alpha}=0$. There might be many other numbers $\tilde r^2_{\alpha}$ satisfying   (\ref{Fact}) in the same group of $r^2_{\alpha}$. Because of (\ref{group2}), $\tilde r^2_{\alpha}/ r^2_{\alpha}$ are bounded as $\alpha\to 1$, but these numbers $\tilde r^2_{\alpha}$ can be ruled out by the following procedure.
Set
\[\tilde A_{2, \alpha}(x) =A_{\alpha} ( x^1_{\alpha}+ r^2_{\alpha}  x).\]
Passing to a subsequence,  $\tilde A_{2, \alpha}$ converges, up-to gauge transformations  away from a finite concentration set of $\{\tilde A_{2, \alpha}\}$,  to a Yang-Mills connection $\tilde A_{2,\infty}$ locally in $\R^4$. Those numbers $\tilde r^2_{\alpha}$ in the same group $r^2_{\alpha}$ have been handled out. If $\{\tilde A_{2, \alpha}\}$ is non-trivial,  then $\{\tilde A_{2, \alpha}\}$ is a new bubble, which is different from the bubble $\{\tilde A_{1, \alpha}\}$.
We must point out that the  above bubble connection  $\tilde A_{2,\infty}$ might be trivial, called a  `ghost bubble'. In this case, there is at least a concentration point $p\in B_2\backslash B_1$ of $\{\tilde A_{2, \alpha}\}$ due to (\ref{Fact}).

 At each concentration point $p$ of $\tilde A_{2, \alpha}$, we can repeat the procedure in Step 1; i.e.  at each concentration point $p$ of $\tilde A_{2, \alpha}$, there are sequences $x_{\alpha}^p\to p$ and $\lambda^p_{\alpha}\to 0$ such that
\[\tilde A_{2, \alpha}(x_{\alpha}^p+\lambda^p_{\alpha}x )\to \tilde A_{2,p,\infty}\] up-to gauge transformation,
where $\tilde A_{2,p,\infty}$ is a Yang-Mills connection on $\R^4$.
Note that $\tilde A_{2,p,\infty}$ is also a bubble for the sequence $\{A_{\alpha}(x^1_{\alpha}+r_{\alpha}^2x_{\alpha}^p+r_{\alpha}^2\lambda_{\alpha}^p x)\}$.

Set $x^{2,p}_{\alpha}=x^1_{\alpha}+r_{\alpha}^2x_{\alpha}^p$. If $p\neq 0$, then
\[
\frac {|x^1_{\alpha}-x^{2,p}_{\alpha}|}{r_{\alpha}^1}= \frac {r_{\alpha}^2} {r_{\alpha}^1}|x_{\alpha}^p| \to \infty  \mbox { as  }\alpha\to 1. \]
Therefore, the bubble $\{\tilde A_{2,p, \infty}\}$ at $p\neq 0$ is different from the bubble $\tilde A_{1,\infty}$.

Since $\lim_{\a\to 1}\frac {r^1_{\alpha}} {r^2_{\alpha}}=0$ and $\tilde A_{1,\infty}$ is  a bubble limiting connection for the sequence  $\{A_{\alpha} ( x^1_{\alpha}+ r^1_{\alpha}  x)=\tilde A_{2, \alpha}(\frac {r^1_{\alpha}} {r^2_{\alpha}}x)\}$,   $p=0$ is a concentration point of $\tilde A_{2, \alpha}$ on $\R^4$. Like Step 1, there is a small $R_2^0>0$ such that the ball $B_{R^2_0}(0)$ contains only the isolated concentration point $0$ of $\tilde A_{2, \alpha}$. Then,
  it can be seen from Step 1 that
\[ \lambda_{\alpha}^0= \frac 1 {\max_{B_{R^2_0}(0)} |F_{\tilde A_{2,\a}}|^{1/2}}\geq  \frac 1 {r^2_{\alpha}\max_{B_{R_0}(x_i)} |F_{A_{\a}}|^{1/2}} =\frac {r^1_{\alpha}} {r^2_{\alpha}}.\]
Similarly to Step 1, the bubble $\tilde A_{2,0,\infty}$  is chosen as the maximal (top) bubble for  $\tilde A_{2, \alpha}$ at $p=0$. Therefore
the bubble $\tilde A_{2,0,\infty}$  must be the same bubble $\tilde A_{1,\infty}$. We can keep it there without a problem.

Next, we must continue the above procedure for possible new multiple bubbles at each blow-up point $p$ again. Since there is a uniform bound $K$ for $YM(A_{\alpha}; M)$ and  each non-trivial bubble on $S^4$  costs at least $\varepsilon_g$ of  the energy by  the gap theorem of  Bourguignon and Lawson \cite{BL}, the above process must stop after finite steps.

\medskip\noindent{\bf Step 3.} To find out all  multiple bubbles.

 Let  $r^3_{\alpha}$ be in the second small group of numbers satisfying (\ref{Fact}) with $\lim_{\alpha \to 1} \frac {r^3_{\alpha}}{r^2_{\alpha}}=\infty$ and
$\lim_{\alpha \to 1}  r^3_{\alpha}=0$.
Set
\[\tilde A_{3, \alpha}(x) =A_{\alpha} (x^1_{\alpha} +r_{\alpha}^3x).\]
 Passing to a subsequence,  $\tilde A_{3, \alpha}$ converges, up-to gauge transformations  away from a finite concentration set of $\{\tilde A_{3, \alpha}\}$, locally to a Yang-Mills connection $\tilde A_{3,\infty}$  in a bundle over $\R^4$.
Then we can repeat the argument of Steps 1-2. All bubbles produced by $\tilde A_{3, \alpha}$, except for those concentrated in $0$,  are different from Steps 1-2.
By induction, we can find out all bubbles in all cases of the finite different groups. The above process must stop after finite steps  by  the gap theorem.

  In summary,  at each group level $k$, the blow-up happens,  there are  finitely many blow-up points  and Yang-Mills bubbles on $\R^4$.
At each level $k$ and each bubble point $p_{k,l}$, there are sequences $x_{\alpha}^{k,l}\to p_{k,l}$ and $r_{\alpha}^k\to 0$ such that passing to a subsequence,
$\tilde A_{\alpha, k,l}( x)=A_{\alpha} (r^{k}_{\alpha}  x+x^{k,l}_{\alpha})$ converges, up-to gauge transformations, to  $\tilde A_{k,l,\infty}$, where $\tilde A_{k,l,\infty}$  is a Yang-Mills connection in a bundle over $\R^4$.

\begin{figure}[h] 
     \begin{center}
        \subfigure{\includegraphics[width=5cm]{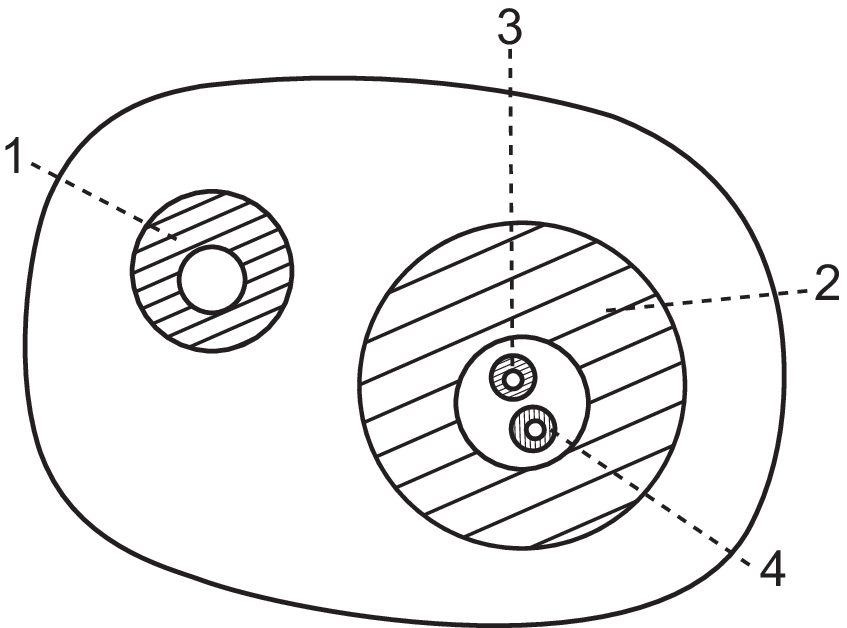}} \qquad \subfigure{\includegraphics[width=5cm]{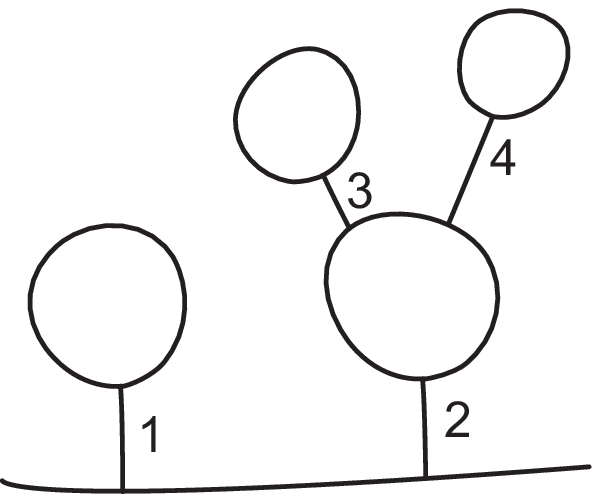}}
     \end{center}
    \end{figure}
(The above procedure of the bubble decomposition
can be illustrated in  the following figures. The above shadow parts in the
figure   stand for the neck regions. The above figures with number 1 can be viewed as  the case of a single bubble and the figures with numbers 2-4 can be seen  as  a special example of three different bubbles, but real cases of multiple bubbles are much more complicated.)

In conclusion, there are  finite   numbers $r^{k}_{\alpha}$, finite points $x^{k,l}_\alpha$, positive  constants $R_{k,l}$, $\delta_{k,l}$ and a finite number  of non-trivial Yang-Mills connections
$\tilde A_{k,l,\infty}$ on $S^4$ such that
\begin{align}\label{neckenergy}
 \lim_{\alpha\to 1} YM(A_{\alpha}; B_{R_0}(x_i))
=&YM(A_{\infty}; B_{R_0}(x_i) )+\sum_{k=1}^L\sum_{l=1}^{J_k}
YM(\tilde A_{k,l, \infty}; S^4)\\
+ &\sum_{k=1}^L \sum_{l=1}^{J_k} \lim_{R_{k,l} \to \infty} \lim_{\delta_{k,l} \to 0} \lim_{\alpha \to 1} YM(\tilde A_{k,l,\alpha}; B_{\delta_{k,l}} \backslash B_{R_{k,l}r^{k}_\alpha}(x^{k,l}_\alpha)).\nonumber
\end{align}
Moreover, at each neck region $B_{\delta_{k,l}} \backslash B_{R_{k,l}r^{k}_\alpha}(x^{k,l}_\alpha)$ in (\ref{neckenergy}), for all $\alpha$ sufficiently close to $1$, we have
 \begin{equation}\label{Basic} \int_{B_{2r} \backslash
B_r (x^{k,l}_{\a})}{|F_{\tilde  A_{k,l,\alpha}}|^2 dV} \leq \varepsilon \end{equation}
 for all $r \in (\frac{R_{k,l}r^k_\alpha}{2},2\delta_{k,l})$, where   $\varepsilon$ is a  fixed constant to be chosen sufficiently small.

Finally, we point out that the above bubble-neck decomposition could be formulated by the method of Parker \cite {Parker}.

 \subsection{No new bubble on each neck region}

According to the previous result on the bubble-neck decomposition (\ref{neckenergy}) with the property (\ref{Basic}), we will prove  no new bubble and energy loss  on each neck region  $B_{\delta_{k,l}} \backslash B_{R_{k,l}r^k_\alpha}(x^{k,l}_{\alpha})$.

For simplicity of notations, we may take $x^k_\alpha$ to be the origin of our coordinate system, and denote  $B_\delta=B_{\delta_{k,l}} (x^k_\alpha)$. Furthermore, if we choose coordinates which are orthonormal at
$x_\alpha$, then in the limit $\delta \to 0$ the metric on $B_\delta$ will approach the Euclidean metric. Therefore, in this section we assume without loss of generality that the
metric on $B_\delta$ is given by $g_{ij} = \delta_{ij}$.

Noting the assumption   (\ref {Basic}) on each annulus $B_{\delta} \backslash B_{Rr_\alpha}(0)$ and applying Lemma \ref{epsilonregularity}, we have that
\begin{equation} \label{supbound}r^{2} \|F \|_{L^\infty(B_{r} \setminus
B_{r/2})(x_{\a})} \leq C\sqrt{\varepsilon} , \quad \forall r\in(Rr_\alpha,\delta) \end{equation} for a sufficiently small constant $\varepsilon>0$.

We choose polar coordinates $(r, \theta)$ on $B_{\delta}$, where $\theta=(\theta^1,\theta^2,\theta^3)$ are coordinates on $S^3_r$. Our connection $1$-form can be decomposed as
 \[ A=A_r+ A_\theta,   \]  where $A_r = \left <A,dr\right>dr$   denotes the radial
part and $A_\theta  =\sum_{j=1}^3 A_{\theta_j}d\theta^j$ denotes the $S^3_r$-part. For a vector field $X=X_k\frac1 {\partial x^k}$, we define \[ X \rfloor F_A  = F_A (X, \, \cdot \, )=\frac 12 X_kF^{ij}dx^j, \] which belongs to $\Gamma(\operatorname{Ad} E \otimes T^*M)$. We write
\[F=F_{A,r \theta}dr\wedge d\theta +\frac 12 F_{\theta_i\theta_j} d\theta_i\wedge d\theta_j.\] Then $\frac{\partial}{\partial r} \rfloor F_A=F_{A,r \theta} d\theta$.

Following Uhlenbeck in \cite{Uhl.removable}, we will construct a broken Hodge gauge. We recall Theorem 2.8 and Corollary 2.9 of \cite{Uhl.removable} in the following:
\begin{Lemma4}
\label{key} Let $D$ be a covariant derivative in a vector bundle $E$ over \[\mathfrak{U}=\{x: 1\leq |x|\leq 2\}.\] There exists a constant $\gamma >0$ such that if
$\|F_A\|_{L^{\infty}(\mathfrak{U})}\leq \gamma $, then there is a gauge transformation in which $D=d+A$, $d^*A = 0$ on $\mathfrak{U}$, $d^*_{\theta} A_{\theta}=0$ on $S_1$ and $S_2$, and
$\int_{S_1}A_r=\int_{S_2}A_r=0$. Moreover, $\|A\|_{L^{\infty}(\mathfrak{U})}\leq C\|F_A\|_{L^{\infty}(\mathfrak{U})}$ and \[(\lambda_4-k' \|F_A\|_{L^{\infty}}^2
)\int_{\mathfrak{U}}|A|^2\leq \int_{\mathfrak{U}}|F_{A}|^2,\] where $\lambda_4$ and $k'$ are positive constants defined in Corollary 2.9 of \cite{Uhl.removable}. \end{Lemma4}

 For $l = \{-1, 0, 1,2,... \}$, we define
\begin{align*} \mathfrak{U}_l & = \{x : 2^{-l-1} \delta \leq |x| \leq 2^{-l} \delta \}, \\ S_l & = \{x:|x| = 2^{-l} \delta \}. \end{align*} We choose $n$ such that \[ B_\delta
\setminus B_{Rr_\alpha} \subseteq \sum_{l=0}^n \mathfrak{U}_l \subseteq B_{\delta} \setminus B_{\frac 1 2 Rr_\alpha}. \] A broken Hodge gauge is a continuous gauge transformation $s$
which satisfies:
 \begin{align*} a) & \quad d^*A_l =0  \quad \mbox{with $A_l = \left. s^*A_\alpha \right|_{\mathfrak{U}_l}$};\\
 b) & \quad \left. A_{l,\theta} \right|_{S_l} = \left. A_{l-1,\theta} \right|_{S_{l}};\\
 c) & \quad d^*_\theta A_{l,\theta} = 0\quad  \mbox{ on $S_l$ and $S_{l+1}$};\\
d) &\quad\int_{S_l}{A_{l,r} dS} = \int_{S_{l+1}}{A_{l,r} dS} = 0,\mbox {where $dS$ is the measure on the boundary $S^3$.}
 \end{align*}

 Since   $r^2\|F_A(x)\|_{L^{\infty}}$  can be chosen sufficiently small, we can apply  Theorem 4.6 of \cite{Uhl.removable} to obtain a broken Hodge gauge on $B_{\delta}\backslash B_{\frac 12 Rr_{\alpha}}$.   We henceforth assume that the
 connection
 $A_\alpha$ is expressed in this gauge.

  We prove the following lemma which will be used in the proof of Lemma \ref{covering}:

 \begin{Lemma3} \label{weitzenbock} Let $B_R$ be the geodesic ball of radius $R$ at a point $x_0 \in M$, and $\phi$ a smooth cut-off function on $B_R$ with $\phi = 1$ on $B_{kR}$,
$k<1$, and $|d\phi| \leq cR^{-1}$. Let $A$ be a connection on $B_R$. There exists a constant $\varepsilon_1>0$ such that if \[ \int_{B_R}{|F_A|^2 dV} < \varepsilon_1, \] then we have \[
\int_{B_{R}}{|\nabla_A F_A|^2} \phi^2 \leq C\int_{B_R}{|D_A^* F_A|^2 \phi^2 dV}+C(1+R^{-2})\int_{B_R}{|F_A|^2 dV}. \] \end{Lemma3} \begin{proof} We compute \begin{align*}
\int_{B_{R}}{|\nabla_A F_A|^2 \phi^2 dV} & \leq \int_{B_{R}}{\left(\left\langle \nabla_A F_A, \nabla_A (\phi^2 F_A)\right\rangle + 2 \phi |d \phi| |F| |\nabla_A F_A| \right)dV} \\ &
\leq \int_{B_R}{ \left( \left\langle \nabla_A^* \nabla_A F_A,\phi^2 F_A \right\rangle + C \phi R^{-1} |F| |\nabla_A F_A| \right) dV}. \end{align*} The second term is handled using
Young's inequality with $\varepsilon$: \[ \phi R^{-1} |F| |\nabla_A F_A| \leq C \varepsilon^{-1} R^{-2} |F|^2 + C \varepsilon |\nabla_A F_A|^2 \phi^2, \] for $\varepsilon > 0$ small. For
the first term, it follows from using the Weitzenb\"ock formula and the Bianchi identity that \begin{align*} &\int_{B_R}{\left\langle \nabla_A^* \nabla_A F_A,\phi^2 F_A \right\rangle
dV} =\int_{B_R}{\left\langle \Delta_A F_A + F_A \# F_A + R_M \# F_A,\phi^2 F_A \right\rangle dV} \\ \leq & \int_{B_R} \left( \left\langle D_A F_A, D_A (\phi^2 F_A) \right\rangle +
\left\langle D_A^* F_A, D_A^* (\phi^2 F_A) \right\rangle  +C \phi^2 |F_A|^2 +c \phi^2 |F_A|^3 \right) dV \\ \leq & \int_{B_R} \left( |D_A^* F_A|^2 \phi^2 +C(1+R^{-2})|F_A|^2\phi^2
+C|F_A|^3\phi^2 +C \phi R^{-1} |F| |D_A^* F_A| \right) dV. \end{align*}
 The last term is handled using Young's inequality. For the cubic term, using the H\"older and Sobolev inequalities, we have
\begin{align*} &\int_{B_R}{|F_A|^3\phi^2 dV}  \leq \left( \int_{B_R}{|F_A|^2 dV} \right)^{1/2} \left( \int_{B_R}{|F_A|^4 \phi^4 dV} \right)^{1/2} \\ & \leq C \left( \int_{B_R}{|F_A|^2
dV} \right)^{\frac 1 2}\left( \int_{B_R}{|F_A|^2\phi^2 dV} + \int_{B_R}{|\nabla_A (F_A\phi)|^2 dV}  \right) \\ & \leq C\varepsilon_1^{\frac 1 2}\left( \int_{B_R}{|F_A|^2\phi^2 dV} +
\int_{B_R}{|\nabla_A F_A|^2 \phi^2 dV} + CR^{-2} \int_{B_R}{|F_A|^2 dV}  \right). \\ \end{align*} Then choosing $\varepsilon_1$ small enough, the result follows. \end{proof}

\begin{Lemma4} \label{covering} There exists a positive constant $\varepsilon_2$ such that if \[\|F_{A_\alpha}\|_{L^\infty(\mathfrak{U}_l)} \leq \varepsilon_2 2^{4l}\delta^{-2},\] and
$\alpha$ is sufficiently close to $1$, we have \[ \int_{B_\delta \setminus B_{Rr_\alpha}}{|A_\alpha||\nabla_{A_\alpha} F_{A_\alpha}| dV} \leq C \int_{B_{2\delta} \setminus B_{\frac 1 4
Rr_\alpha}}{|F_{A_\alpha}|^2 dV}. \] \end{Lemma4} \begin{proof}
 It follows from using H\"older's inequality that
\begin{align*} \int_{B_\delta \setminus B_{Rr_\alpha}}{|A_\alpha||\nabla_{A_\alpha} F_{A_\alpha}| dV} & \leq  \sum_{l=0}^n \int_{\mathfrak{U}_l}{|A_\alpha||\nabla_{A_\alpha}
F_{A_\alpha}| dV}  \\ & \leq \sum_{l=0}^n \left( \int_{\mathfrak{U}_l}{|A_\alpha|^2 dV} \right)^{\frac 1 2} \left( \int_{\mathfrak{U}_l}{|\nabla_{A_\alpha} F_{A_\alpha}|^2 dV}
\right)^{\frac 1 2}. \end{align*} For the first factor, scaling in Lemma \ref{key} (see also Corollary 2.9 of \cite{Uhl.removable}), we have \begin{align} \int_{\mathfrak{U}_l}
{|A_{\a}|^2 dV} & \leq C2^{-2l}\delta^{2}(\lambda_4-k' 2^{-4l}\delta^{2} \|F_{A_\alpha}\|_{L^{\infty}(\mathfrak{U}_l)}^2)^{-1}\int_{\mathfrak{U}_l} {|F_{A_\alpha}|^2 dV} \nonumber \\ &
\leq C2^{-2l}\delta^{2}\int_{\mathfrak{U}_l} {|F_{A_\alpha}|^2 dV}, \label{firstfactor} \end{align} after choosing $\varepsilon_2$ sufficiently small. To deal with the second factor,
we cover $\mathfrak{U}_l$ with finitely many open balls $B_{2R/3}(i)$ such that $\cup_i B_R(i) \subset W_l = \mathfrak{U}_{l-1} \cup \mathfrak{U}_l \cup \mathfrak{U}_{l+1}$, for some $R$ with
$\frac {3\delta}{2^{l+3}}\leq R<\frac {3\delta}{2^{l+2}}$. We let $\phi_i$ be a smooth cut-off function on $B_R(i)$ with $\phi = 1$ on $B_{2R/3}(i)$ and $|d\phi_i| \leq CR^{-1}$.  Note
that from (\ref{eulerlagrange2}), we have \[ |D^*_{A_\alpha} F_{A_\alpha}| \leq C(\alpha-1)|\nabla_{A_\alpha} F_{A_\alpha}|. \] Applying Lemma \ref{weitzenbock} to the ball $B_R(i)$
yields \begin{align*} &\quad\int_{B_R(i)}{|\nabla_{A_\alpha} F_{A_\alpha}|^2 \phi_i^2 dV} \\ & \leq C\int_{B_R(i)}{|D_{A_\alpha}^* F_{A_\alpha}|^2 \phi_i^2 dV}+C(1+R^{-2})
\int_{B_R(i)}{|F_{A_\alpha}|^2 dV} \\ & \leq C(\alpha - 1)\int_{B_R(i)}{|\nabla_{A_\alpha} F_{A_\alpha}|^2 \phi_i^2 dV}+C(1+2^{2l}\delta^{-2}) \int_{B_R(i)}{|F_{A_\alpha}|^2 dV}.
\end{align*} Then for $\alpha$ sufficiently close to $1$, we find \[ \int_{B_R(i)}{|\nabla_{A_\alpha} F_{A_\alpha}|^2 \phi_i^2 dV} \leq C(1+2^{2l}\delta^{-2})
\int_{B_R(i)}{|F_{A_\alpha}|^2 dV}. \] Summing over $i$, it follows that \[ \int_{\mathfrak{U}_l}{|\nabla_{A_\alpha} F_{A_\alpha}|^2 dV} \leq C(1+2^{2l}\delta^{-2})
\int_{W_l}{|F_{A_\alpha}|^2 dV}. \] Finally, combining the above we have \begin{align*} &\quad \int_{B_\delta \setminus B_{Rr_\alpha}}{|A_\alpha||\nabla_{A_\alpha} F_{A_\alpha}| dV}\\
&\leq C\sum_{l=0}^n \left(2^{-2l}\delta^2\int_{\mathfrak{U}_l} {|F_{A_\alpha}|^2 dV } \right)^{\frac 1 2} \left((1+2^{2l}\delta^{-2}) \int_{W_l}{|F_{A_\alpha}|^2 dV} \right)^{\frac 1
2},
 \end{align*}
which implies the desired result. \end{proof}

\begin{Lemma5}  \label{boundary} There are a sequence $\delta\to 0$ and  a subsequence $\alpha_k \to 1$ such that  if $\int_{B_{2r} \setminus B_r}|F_A|^2 \leq \varepsilon_3$ for all $r\in( \frac  12 Rr_\alpha ,2\delta)$  and for a sufficiently small $\varepsilon_3 > 0$, the boundary integral satisfies
\[ \lim_{\delta \to 0} \lim_{\alpha_k \to 1}\int_{\partial B_\delta}{\left\langle A_{\alpha_k,\theta}, F_{A_{\alpha_k}, r
\theta} \right\rangle dS} = 0. \] \end{Lemma5}
 \begin{proof} Since $YM(A_{\alpha}) \leq K$ for each $\alpha$, it follows from Fatou's lemma that
  \[ \int_0^{R_0}\liminf_{\alpha \to 1} \int_{\partial
B_r} {|F_{A_\alpha, r \theta}|^2} dS dr \leq   K. \]
It follows that \begin{equation} \label{tozero} \lim_{\delta \to 0} \left( \delta \liminf_{\alpha \to 1}
\int_{\partial B_{\delta}} {|F_{A_\alpha, r \theta}|^2} dS \right) = 0. \end{equation} From Theorem 4.6 of \cite{Uhl.removable} (see also Lemma \ref{key}), in the Hodge gauge, we have \[ \delta \| A_{\alpha,\theta}
\|_{L^\infty(\partial B_\delta)} \leq C \delta^2 \| F_{A_\alpha} \|_{L^\infty(B_{\delta}\backslash B_{\frac{\delta}2 })} \leq C. \] This implies \[ \int_{\partial B_\delta} {| A_{\alpha,\theta}|^2}dS \leq
C\delta. \] By H\"older's inequality, we obtain \begin{align*} \int_{\partial B_\delta}{\left\langle A_{\alpha,\theta}, F_{A_\alpha, r \theta} \right\rangle dS} & \leq \left(
\int_{\partial B_\delta} {|A_{\alpha,\theta}|^2dS} \right)^{\frac 1 2} \left( \int_{\partial B_\delta} {|F_{A_\alpha, r \theta}|^2dS} \right)^{\frac 1 2} \\ & \leq C \left( \delta
\int_{\partial B_\delta} {|F_{A_\alpha, r \theta}|^2dS} \right)^{\frac 1 2}. \end{align*} Choosing a sequence $\delta\to 0$ and a suitable subsequence of $\alpha_k\to 1$, the claim follows from (\ref{tozero}). \end{proof}

\begin{Lemma7} \label{Removable} There exists a sufficiently small $\varepsilon > 0$ such that if $\int_{B_{2r} \setminus B_r}|F_A|^2 \leq \varepsilon$ for all $r\in(\frac  12 Rr_\alpha,2\delta)$, then
\[ \lim_{R \to \infty} \lim_{\delta \to 0} \lim_{\alpha \to 1} \int_{B_\delta \setminus B_{Rr_\alpha}}{|F_{A_\alpha}|^2 dV} = 0, \] where   the subsequence $\alpha \to 1$ is
suitably chosen. \end{Lemma7}
\begin{proof} Under the assumption,
(\ref{supbound}) holds, so the above broken Hodge gauge exists by choosing $\varepsilon$ sufficiently small. On each annulus $\mathfrak{U}_l$, we calculate \begin{align} &\int_{\mathfrak{U}_l} {|F_{A_\alpha}|^2 dV} =
\int_{\mathfrak{U}_l} {\left\langle dA_\alpha + [A_\alpha, A_\alpha] ,F_{A_\alpha} \right\rangle dV} \nonumber \\ = & \int_{\mathfrak{U}_l} {\left\langle D_{A_\alpha}A_\alpha +
A_\alpha\# A_\alpha, F_{A_\alpha} \right\rangle dV} \nonumber \\ = & \int_{\mathfrak{U}_l} {\left\langle A_\alpha ,D_{A_\alpha}^*F_{A_\alpha} \right\rangle dV} +
\int_{\mathfrak{U}_l} {\left\langle A_\alpha\# A_\alpha,F_{A_\alpha} \right\rangle dV} \nonumber \\ & + \int_{S_l}{\left\langle A_{\alpha,\theta}, F_{A_\alpha, r \theta} \right\rangle
dS} - \int_{S_{l+1}}{\left\langle A_{\alpha,\theta}, F_{A_\alpha, r \theta} \right\rangle dS}, \label{calculation} \end{align}
where $A_\alpha\# A_\alpha$ denotes a multi-linear combination of $A_{\alpha}$ and $A_{\alpha}$. The boundary terms here can be derived by constructing a
radial cut-off function $\phi(r)$ with $\phi = 1$ on $\mathfrak{U}_l$, and $\phi = 0$ outside of a slightly larger annulus \[ \mathfrak{U'}_l = \{x : 2^{-l-1}\delta  - \varepsilon \leq
|x| \leq 2^{-l} \delta + \varepsilon \}. \] The derivative is $d \phi = \frac 1 \varepsilon dr$ on $\mathfrak{U'}_l \setminus \mathfrak{U}_l$, and zero elsewhere. Then \[
\int_{\mathfrak{U'}_l} {\left\langle D_{A_\alpha}A_\alpha, F_{A_\alpha} \right\rangle \phi dV} = \int_{\mathfrak{U'}_l} {\left\langle A_\alpha, D_{A_\alpha}^* F_{A_\alpha}
\right\rangle \phi dV} - \int_{\mathfrak{U'}_l} {\left\langle d\phi \wedge A_\alpha, F_{A_\alpha} \right\rangle dV}, \] where the last term will become the boundary terms of
(\ref{calculation}) in the limit $\varepsilon \to 0$. Summing (\ref{calculation}) over $l$ and using equation (\ref{eulerlagrange2}), we find \begin{align*} & \sum_{l=0}^n
\int_{\mathfrak{U}_l} {| F_{A_\alpha}|^2 dV} =  \sum_{l=0}^n \int_{\mathfrak{U}_l} {\left\langle A_\alpha\# A_\alpha, F_{A_\alpha} \right\rangle dV} + \int_{S_0}{\left\langle
A_{\alpha,\theta}, F_{A_\alpha, r \theta} \right\rangle dS} \\ & \quad - \int_{S_{n+1}}{\left\langle A_{\alpha,\theta}, F_{A_\alpha, r \theta} \right\rangle dS} + (\alpha-1)
\sum_{l=0}^n \int_{\mathfrak{U}_l}{\left\langle A_\alpha, \frac{*(d|F_{A_\alpha}|^2\wedge*F_{A_\alpha})}{1+|F_{A_\alpha}|^2} \right\rangle dV}. \end{align*} Choosing a suitable subsequence of $\alpha\to 1$, two boundary terms are
going to zero by Lemma \ref{boundary}. Using Lemma \ref{covering}, the final term is bounded by \[ C (\alpha - 1) \sum_{l} \int_{\mathfrak{U}_l}{|A_\alpha||\nabla_{A_\alpha}
F_{A_\alpha}| dV} \leq  C(\alpha - 1)\int_{B_{4\delta} \setminus B_{\frac 1 4 Rr_\alpha}}{|F_{A_\alpha}|^2 dV}, \] which is going to zero as $\alpha \to 1$. For the first term,
recalling (\ref{firstfactor}), \begin{align*} \int_{\mathfrak{U}_l}{|\left\langle A_\alpha\# A_\alpha, F_{A_\alpha} \right\rangle |} & \leq C\|F_{A_\alpha}\|_{L^\infty(\mathfrak{U}_l)}
\int_{\mathfrak{U}_l} {|A_\alpha|^2 dV} \\ & \leq C2^{-2l}\delta^2 \|F_{A_\alpha}\|_{L^\infty(\mathfrak{U}_l)} \int_{\mathfrak{U}_l} {|F_{A_\alpha}|^2 dV}. \end{align*} Note that
$2^{-2l}\delta^2 \|F_{A_\alpha}\|_{L^\infty(\mathfrak{U}_l)}$ can be made sufficiently small, so that this term may be absorbed into the left hand side. The required result follows
from Lemma \ref {boundary} and the uniform bound of $\int_M |F_{A_{\a}}|^2\,dV$. \end{proof}

We now complete the proof of Theorem \ref{maintheorem}. \begin{proof} [Proof of Theorem \ref{maintheorem}] Theorem \ref{maintheorem} follows from Lemma \ref{Removable} and the bubble-neck decomposition (\ref{neckenergy}) by choosing $\varepsilon$ sufficiently small. \end{proof}

 By letting $\alpha_i = 1$ for all $i$,    Theorem \ref{maintheorem} yields  a new proof of a result of Rivi\`ere \cite{Riviere} on sequences of Yang-Mills connections:

\begin{Corollary} \label{corol}
    Let  $A_i$ be a sequence of smooth Yang-Mills connections on $E$.  Then there exists a finite set $S\subset M$, and a sequence of gauge transformations $s_i$ defined on
    $M\setminus
    S$, such that for any compact $K\subset M\setminus S$, $s_i^*{A}_i$ converges to $A_{\infty}$ smoothly in $K$. Moreover, there are a finite number of bubble bundles
    $E_1,\cdots,E_l$ over $S^4$ and Yang-Mills connections $\tilde{A}_{1,\infty},\cdots,\tilde{A}_{l,\infty}$ such that
    \begin{equation*}
        \lim_{i\to\infty} YM(A_i)=YM(A_{\infty})+\sum_{i=1}^l YM(\tilde{A}_{i,\infty}).
    \end{equation*}
\end{Corollary}

As a second  consequence of Theorem \ref{maintheorem}, we can also give a simple proof of the energy identity for a minimizing sequence (see \cite {Isobe1}, \cite {HTY}):

 \begin{Proposition} \label{prop}
    Let  $E$ be a vector bundle over $M$. Assume that $A_i$ is a minimizing sequence of the Yang-Mills functional $YM$ among smooth connections on $E$, which converges weakly to some
    limit connection $A_{\infty}$ by Sedlacek's result \cite{Sedlacek.direct}.
Then there are a finite number of bubble bundles $E_1,\cdots,E_l$ over $S^4$ and Yang-Mills connections $\tilde{A}_{1,\infty},\cdots,\tilde{A}_{l,\infty}$ such that
    \begin{equation*}
        \lim_{i\to\infty} YM(A_i)=YM(A_{\infty})+\sum_{i=1}^l YM(\tilde{A}_{i,\infty}).
    \end{equation*}
\end{Proposition} \begin{proof} We choose a sequence $\alpha_i\to 1$, and let $A_{\alpha_i}(t)$ be a solution to the Yang-Mills $\alpha_i$-flow of \cite{HTY}  with initial condition
$A_{\alpha_i}(0)=A_i$ (see a similar argument for harmonic maps in  \cite {HongYinSU}). Then $A_{\alpha_i}$ is a smooth solution to the flow equation
 \begin{equation}\label{eqn:alphaflow}
  \frac{\partial A_{\alpha_i}}{\partial t}=-D_{A_{\alpha_i}}^* F_{A_{\alpha_i}}+(\alpha_i-1)\frac{*(d|F_{A_{\alpha_i}}|^2\wedge *F_{A_{\alpha_i}})}{1+|F_{A_{\alpha_i}}|^2}.
\end{equation} Since $A_i$ is a minimizing sequence, by suitably choosing the sequence $\alpha_i\to 1$, there exists at least one $t_0\in [1/2,1]$ such that for all $i$,
    \begin{equation}
        \int_M |\partial_t A_{\alpha_i}(\cdot,t_0)|^2 dV \to 0.
    \end{equation}
The result then follows from the same arguments as for Theorem 2. See \cite{HTY} for more details. \end{proof}

\subsection{Applications to the Yang-Mills flow}

In order to prove Theorem 3,  we need the following local energy inequality:
\begin{Lemma6}  \label{Local}
 Let $A(t)$ be a solution to the Yang-Mills flow (\ref {YMHF}) in $M\times[0,T)$ for a $T>0$ with initial value $A(0)=A_0$. For any $x_0$ with $B_{2R}(x_0)\subset X$
and  for any two  $s,\tau \in [0, T)$ with $s<\tau$, we have
\begin{align*}
 \quad \int_{M}|F_A|^2 (\cdot ,s)\,dV+ \int_0^s \int_M|\partial_t
A|^2\,dV\,dt \leq \int_{M}|F_{A_0}|^2 \,dV, \end{align*}
\begin{align*}  \int_{B_{R}(x_0) }|F_A|^2 (\cdot ,\tau)\,dV\leq &\,\int_{B_{2R}(x_0)} |F_A|^2 (\cdot ,s)|^2\,dV
+  \frac {C(\tau-s)}{R^2}\,\int_M |F_{A_0}|^2\,dV  \end{align*}
and
\begin{align*}  \int_{B_{R}(x_0) }|F_A|^2 (\cdot ,s)\,dV\leq &\,\int_{B_{2R}(x_0)} |F_A|^2 (\cdot , \tau )\,dV
+C\int_{s}^{\tau }\int_{B_{2R}(x_0)} |\partial_t A|^2\,dV\,dt\\ +&C\left (\frac {(\tau -s)}{R^2}\,\int_{B_{2R}(x_0)} |F_{A_0}|^2\,dV\, \int_{s}^{\tau
}\int_{B_{2R}(x_0)}|\partial_t A|^2\,dV\,dt \right )^{1/2}\,. \end{align*}
\end{Lemma6}
\begin{proof} The proof can be found in \cite{StruweYMheat} and  Lemma 5 of \cite{HT}. \end{proof}
Now we present a proof of Theorem \ref{thm3}. \begin{proof}[Proof of Theorem \ref{thm3}] Let $T$ be the maximal existence time of the Yang-Mills flow.  By the result of Struwe in \cite
{StruweYMheat}, there are a finite number of points $\{x_1,..., x_l\}$ such that $A(t)$ converges, up-to gauge transformations,  to a connection $A(T)$  in $C^{\infty}_{loc}(M\backslash\{x_1,...,
x_l\})$ as  $t\to T$. In a local trivialization of $E$ near each singularity $x_j$, $D(t)$ can be expressed by $d+ A(t)$ with $A(t)=A_i(x,t) dx^i$. At each singularity $x_j$, there is a
$R_0>0$ such that there is no singularity inside $B_{R_0}(x_j)$. Then there is a $\Theta>0$ such that as $t\to T$
\begin{equation}\label{limit}  |F_{A(t)}|^2 dV\to \Theta \delta_{x_j} + |F_{A(T)}|^2 dV,\end{equation}
 where
$\delta_{x_j}$ denotes the Dirac mass at the singularity $x_j$. This can be proved by using Lemma \ref{Local}   (e.g. \cite {LW}). Then there exist sequences $r_k\to 0$, $t_k\to T$ such that as
$k\to\infty$,
\begin{equation}\label{dense}\lim_{k\to \infty} \int_{B_{r_k}(x_j)}|F_{A(t_k)}|^2 \,dV=\Theta .\end{equation}
We consider the rescaled connection \[ D_{\tilde A_k}(t)=d
+\tilde A_k(x,t):= d + r_kA_i(x_j+r_kx, t_k+r_k^2 t)dx^i.\] Then $\tilde A_k(t)$ satisfies
\begin{equation}\label{YMF}\frac {\partial \tilde A}{\partial t}=-D_{\tilde A}^*F_{\tilde
A}\quad  \mbox{in  }B_{r_k^{-1}(0)}\times [-2, 0]\end{equation}
and  \begin{equation} \int_{-2}^{0} \int_{B_{r_k^{-1}}(0)} |\partial_t \tilde A_k|^2=\int^{t_k}_{t_k-2r_k^2}
\int_{B_{r_k^{-1}}(0)} |\partial_t A(t)|^2\leq \int^{t_k}_{t_k-2r_k^2} \int_{M} |\partial_t A(t)|^2\to 0  \end{equation}
 as $k\to \infty$.
Then there is a $\tilde t_k\in (-1, 0)$ such that
\begin{equation}\label{GR}
 \int_{B_{r_k^{-1}}(0)} |\partial_t \tilde A_k(\tilde t_k)|^2\to 0,  \quad \lim_{k\to \infty} \int_{B_{1}(0)}|F_{\tilde A_k}(\cdot ,0)|^2 =\Theta . \end{equation}
 By applying Lemma \ref{Local} again, we have
\begin{align*} \int_{B_{R}(0)} |F_{\tilde A_k}(\cdot ,\tilde t_k)|^2 \geq &\int_{B_{1}(0)}|F_{\tilde A_k}(\cdot ,0)|^2 -\frac {C}{R^2}\,\int_M |F_{A_0}|^2\,dV\, \end{align*}
 for $R>2$. This implies
\[ \lim_{k\to\infty}\int_{B_{R}(0)} |F_{\tilde A_k(\tilde t_k)}|^2\geq  \Theta .\]
By (\ref{limit}), we know
\[\lim_{k\to\infty}\int_{B_{R}(0)} |F_{\tilde A_k(\tilde
t_k)}|^2= \Theta.\]
Since $\tilde A_k(t)$ satisfies the Yang-Mills flow (\ref{YMF}), a parabolic $\varepsilon$-regularity estimate holds; i.e. there is a constant $ \varepsilon>0$  such that if
 \begin{equation}\label{small-sup}\sup_{\tilde
t_k-4r^2\leq t\leq \tilde t_k}\int_{B_{r}(0)}|F_{\tilde A_k}|^2dV<  \varepsilon  \end{equation} for some small $r>0$,
then we have
\[  |F_{\tilde A_k}(0, \tilde
t_k)|^2 \leq \frac C {r^6} \int_{P_r(\tilde
t_k)}|F_{\tilde A_k}(
t)|^2\,dV\,dt\leq \frac {C\varepsilon} {r^4},\]
where $P_r(\tilde
t_k)=B_r(0)\times (\tilde
t_k-r^2, \tilde
t_k]$. For the above $\varepsilon$ in  (\ref{small-sup}),  there is a constant $R$  satisfying
 \begin{equation}  \int_{B_{2r}(0)}|F_{\tilde A_k}(t_k)|^2dV< \varepsilon .  \end{equation} Using (\ref{GR})  we  apply  Lemma \ref{Local}   to obtain that
\begin{align*}\int_{B_{r}(0)}|F_{\tilde A_k}|^2dV\leq &C\int_{B_{2r}(0)}|F_{\tilde A_k}(\cdot ,\tilde t_k)|^2 +C\int_{\tilde t_k-4r^2}^{\tilde t_k }\int_{B_{2r}(0)} |\partial_t \tilde A_k|^2\,dV\,dt\\  &+C\left (\,\int_{M} |F_{A_0}|^2\,dV\, \int_{\tilde t_k-4r^2}^{\tilde t_k }
\int_{B_{2r}(0)}|\partial_t \tilde A_k|^2\,dV\,dt \right )^{1/2}\,\\
&<C\varepsilon \end{align*}
for sufficiently large $k$.
With these estimates, the bubble tree procedure also works for the sequence of connections $A_k(\tilde t_k)$ on each $B_R(0)$. Therefore, using the proof of Theorem \ref{maintheorem}  on each $B_R(0)$, there is a finite number $l_R$ bubbling Yang-Mills connections $\tilde A_{i, R}$ on $S^4$ such that
 \begin{equation*}\lim_{k\to\infty}\int_{B_{R}(0)} |F_{\tilde
A_k(\tilde t_k)}|^2=\sum_{i=1}^{l_R}YM(\tilde A_{i, R}; S^4). \end{equation*}
By using \cite {BL}, there is a constant $\varepsilon_0>0$ such that any non-trivial Yang-Mills connection $\tilde A_{i,
R}$ on $S^4$ has \[YM(A_{i, R}; S^4)\geq \varepsilon_0,\] which implies $1\leq l_R\leq [\frac {\Theta}{\varepsilon_0}]$. As $R\to \infty$, it follows from Corollary \ref{corol} that
\begin{align*} \Theta &=
\lim_{R\to \infty}\lim_{k\to\infty}\int_{B_{R}(0)} |F_{\tilde A_k(\tilde t_k)}|^2\\
&=\lim_{R\to\infty}\sum_{i=1}^{l_R}YM(A_{i, R}; S^4)=\sum_{j=1}^{l}YM(\tilde A_{j,\infty}; S^4), \end{align*}
where each $\tilde A_{j,\infty}$ is a Yang-Mills connections in a vector bundle $E_j$ over  $S^4$. This proves Theorem 3.
\end{proof}

Finally, we make a remark about the case of $T=\infty$.  Let $E$ be a complex vector bundle over a $4$-manifold $M$. Let $A$ be a global  smooth solution of the
Yang-Mills flow equation (\ref{YMHF}) in $M\times [0,\infty )$ with
initial value $A_0$. For any a sequence $\{t_k\}$ with $t_k\to\infty$, there is a
subsequence, still denoted by $\{t_k\}$, such that as
$k\to\infty$, $A(x,t_k)$ converges in $C^{\infty}(M\backslash
\Sigma)$
 to a solution $A_{\infty}$ of the Yang-Mills  equation (1.2), where $\Sigma$   is a finite   set of singularities in $M$. At each singularity $x_i$,
  \[
\liminf_{k\to\infty}\,\int_{B_r(x_i)}  |F_A|^2(\cdot,
t_k)|^2\,dV\geq \varepsilon_0  \]
for a constant $\varepsilon_0>0$. The second Chern classes (e.g. \cite {Tian.calibrated}) is defined by
\[ C_2(E)=\frac 1{8\pi^2}[\mbox {tr}(F_{A}\wedge F_{A})-\mbox {tr} F_{A}\wedge   \mbox {tr} F_{A}].\]

Theorem \ref{thm3} implies that
\begin{align*} C_2(E; M)&=\frac 1{8\pi^2}\lim_{k\to\infty}\int_M [\mbox {tr}(F_{A(t_k)}\wedge F_{A(t_k)})-\mbox {tr} F_{A(t_k)}\wedge   \mbox {tr} F_{A(t_k)}]\\
&
=\frac 1{8\pi^2} \int_M [\mbox {tr}(F_{A_{\infty}}\wedge F_{A_{\infty}})-\mbox {tr} F_{A_{\infty}}\wedge   \mbox {tr} F_{A_{\infty}}]\\
&+\frac 1{8\pi^2} \sum_{i=1}^l\int_{S^4} [\mbox {tr}(F_{\tilde A_i}\wedge F_{\tilde A_i})-\mbox {tr} F_{\tilde A_i}\wedge   \mbox {tr} F_{\tilde A_i}]\\
&=C_2(E_{\infty}; M)+\sum_{i=1}^lC_2(E_{i}; S^4) ,
\end{align*}
where $C_2(E_{\infty}; M)$ is the second Chern number of the limiting bundle induced by $A_{\infty}$, and $C_2(E_{i}; S^4)$ is the second Chern number of the bubbling bundle $E_i$ over $S^4$.

\section{Appendix: Regularity}

In this section we prove that a Yang-Mills $\alpha$-connection, a weak solution to the Euler-Lagrange equation (\ref{eulerlagrange1}), is gauge-equivalent to a smooth solution.
  We say that $A \in \mathcal{A}^{1,2 \alpha}$ is a weak solution to (\ref{eulerlagrange1}) if
\[ \int_M {\left( 1 + |F_A|^2 \right)^{\alpha - 1}\left\langle F_A,D_A B\right\rangle dV}=0 \] for any $B \in C^{\infty}(\mathcal A)$; i.e. $A$ satisfies

 \begin{equation}\label {5.1}
 d^*_A\left ((1+|F_A|^2)^{\alpha -1} F_A\right )=0 \quad \mbox { in } U \end{equation}
 in the weak sense.

The proof is essentially similar  to the one  in \cite{Isobe}   for $p$-Yang-Mills connections (see a similar approach  in \cite {GH}). Since there are some differences, we would like to outline the main points.
 Let $x_0$ be a point in $M$. For any $\varepsilon_0$, there is a sufficiently small $R_0$ so that for each $R>0$ with $R\leq R_0$,
 \[\int_{B_{R}(x_0)} |F_A|^2\leq \int_{B_{R_0}(x_0)}|F_A|^2\leq \varepsilon_0.\]
For a sufficiently small  $\varepsilon_0$, there is a gauge transformation $\sigma$ (see \cite {U2}) such that $\sigma^* (D_A)=d+A$ with $d^*A=0$ in $B_{R}(x_0)$, $A\cdot \nu =0$ on $\partial  B_R(x_0)$, and
 \[\int_{B_{R}(x_0)} R^{-2\alpha} |A|^{2\alpha} +|\nabla
A|^{2\alpha} \leq \int_{B_{R}(x_0)}|F_A|^{2\alpha}.\]
Similarly to Lemma 2.1 in \cite {Isobe}  there is a $B$
such that
\begin{align}\label {5.1}
&d^*\left ((1+|dB|^2)^{\alpha -1} dB\right )=0\quad \mbox { in }  B_{R}(x_0),\\
&d^*B=0  \quad \mbox { in }  B_{R}(x_0),\\
& i^*B=i^*A \quad \mbox{ on } \partial B_{R}(x_0).\end{align}
By the result of \cite{U1}, we have
\begin{align*}
 \int_{B_{\rho} (x_0)} |d B|^{2\alpha}dV \leq C \left (\frac {\rho}{R}\right )^4 \int_{B_R(x_0)}  |dB|^{2\alpha } dV\leq C \left (\frac {\rho}{R}\right )^4 \int_{B_{R}(x_0)}  |dA|^{2\alpha }  dV \end{align*}
 for any $\rho<R\leq R_0$.
To compare the above equations of $A$ and $B$, we have

\begin{align*}
& \quad \int_{B_{R}(x_0)} |F_A-dB|^{2\alpha}dV\\
& \leq C \int_{B_{R}(x_0)} \left < (1+|F_A|^2)^{\alpha -1} F_A-(1+|dB|^2)^{\alpha -1} dB, F_A-dB\right> \,dV\\
&  \leq C \int_{B_{R}(x_0)}  \left | (1+|F_A|^2)^{\alpha -1} F_A-(1+|dB|^2)^{\alpha -1} dB\right ||A|^2 \,dV\\
& \leq \varepsilon  \int_{B_{R}(x_0)}  (1+|F_A|^2)^{\alpha}\,dV + C \int_{B_{R} (x_0)}(1+ |A|^{4\alpha})\,dV
 \end{align*}
 for a sufficiently small $\varepsilon$, which will be determined later.
 By the H\"older and Sobolev inequalities, we have
\begin{align*}
&\int_{B_{R} (x_0)} |A|^{4\alpha} \,dV\leq CR^{4(\alpha -1)} \left (\int_{B_{R} (x_0)} |A|^{4\alpha \frac 1 {2-\alpha}}\,dV \right )^{2-\alpha}\\
&\leq CR^{4(\alpha -1)} \left (\int_{B_{R} (x_0)} R^{-2\alpha} |A|^{2\alpha} +|\nabla A|^{2\alpha}\,dV\right )^{2}\\
&\leq CR^{4(\alpha -1)} \left (\int_{B_{R}(x_0)}  |F_A|^{2\alpha}\,dV\right )^2\leq C_1R^{4(\alpha -1)}\int_{B_{R}(x_0)}  |F_A|^{2\alpha}\,dV.
 \end{align*}
Combining the above two estimates, we can get that for any $\rho<R\leq R_0$,
\begin{align}\label {4.2}\int_{B_{\rho}(x_0)} (1+|F_A|^2)^{\alpha}\,dV \leq C\left [\left (\frac
{\rho}{R}\right )^{4} +\varepsilon + R^{4(\alpha -1)} \right ]\int_{B_{R}(x_0)} (1+|F_A|^2)^{\alpha}\,dV +CR^4. \end{align}
For a sufficiently small $R_0$ and $\varepsilon$, a well-known lemma (see Lemma 2.1 of Chapter III in \cite {G}) implies that for any small constant $\delta >0$, there is a constant $C$ such that
\begin{align*}\int_{B_{\rho}(x_0)} |F_A|^{2\alpha}\leq \int_{B_{\rho}(x_0)} (1+|F_A|^2)^{\alpha} \leq C\left (\frac
{\rho}{R}\right )^{4-\delta } \int_{B_{R}(x_0)} (1+|F_A|^2)^{\alpha}+ C \rho^{4-\delta } \end{align*}
for any $\rho<R\leq R_0$.

For a sufficiently small  $R_0$, there is a gauge transformation $\sigma$ (see \cite {U2}) such that $\sigma^* (D_A)=d+A$ with $d^*A=0$ in $B_{R_0}(x_0)$, $A\cdot \nu =0$ on $\partial  B_{R_0}(x_0)$, and
 \[\int_{B_{R_0}(x_0)} R_0^{-2\alpha} |A|^{2\alpha} +|\nabla
A|^{2\alpha} \leq \int_{B_{R_0}(x_0)}|F_A|^{2\alpha}.\]

For each $R\leq R_0$ we  use Lemma 2.3 of  \cite{Isobe} to show that there is a $B\in W^{1,p} (B_R(x_0))$ such that
\[dB=0, \quad d^*B=0 \quad \mbox{in } B_R(x_0)\]
with $B\cdot \nu=A\cdot \nu$ on $\partial B_R(x_0)$. Then we have
\[ \int_{B_{\rho} (x_0)} |B|^{4\alpha} \leq C \left (\frac {\rho}{R}\right )^4 \int_{B_R(x_0)}  |B|^{4\alpha }.\]
Since $(A-B)\cdot \nu =0$ on $\partial B_R(x_0)$ and $d^*(A-B)=0$ in $B_R(x_0)$, we can apply  Lemma 2.5 of \cite {U2} and the Sobolev embedding theorem to get
\[ \int_{B_R(x_0)}  |A-B|^{4\alpha }\leq C R^{4\a -4}\left ( \int_{B_R(x_0)}|dA-dB|^{2\alpha}\right )^2\leq C \left ( \int_{B_R(x_0)}|dA|^{2\alpha}\right )^2.  \]
This implies that for any $\rho< R\leq R_0$,
\begin{align}\label {4.3}
 \int_{B_{\rho} (x_0)} |A|^{4\alpha} \leq C \left (\frac {\rho}{R}\right )^4 \int_{B_R(x_0)}  |A|^{4\alpha } + C\left ( \int_{B_R(x_0)}|dA|^{2\alpha}\right )^2.   \end{align}
From (\ref{4.2}) and (\ref{4.3}), we get
 \begin{align*}
 &\int_{B_{\rho} (x_0)} |dA|^{2\alpha}+|A|^{4\alpha}\,dV\\
  \leq &C \left (\left (\frac {\rho}{R}\right )^4 + \int_{B_R(x_0)}|dA|^{2\alpha}+\varepsilon + R^{4(\alpha -1)}\right )\int_{B_R(x_0)} [ |dA|^{2\alpha}+|A|^{4\alpha }]\,dV\\
  &+CR^4 .  \end{align*}
For each small $\varepsilon$, there is a sufficiently small $R_0$ such that
\[\int_{B_{R_0}(x_0)}|\nabla A|^2\leq C \int_{B_{R_0}(x_0)}|F_A|^2\leq \varepsilon.\]
Then applying Lemma 2.1 of Chapter III in \cite {G} again, we have
\begin{align*}\int_{B_{\rho}(x_0)} |dA|^{2\alpha}+|A|^{4\alpha}\leq   C \rho^{4-\delta } \end{align*}
for a small $\delta >0$ and for any $\rho\leq  R_0$. Then it follows from Lemma 2.3 of \cite{U2} that
 \begin{align*}
 &\int_{B_{\rho} (x_0)} |\nabla A|^{2}\,dV
  \leq C \left (\frac {\rho}{R}\right )^4 \int_{B_R(x_0)}|\nabla A|^{2}+C  \int_{B_R(x_0)} |\nabla (A-B)|^2\,dV\\
  &\leq C \left (\frac {\rho}{R}\right )^4 \int_{B_R(x_0)}|\nabla A|^{2}+ C \int_{B_R(x_0)} |dA|^2\,dV\\
  &\leq C \left (\frac {\rho}{R}\right )^4 \int_{B_R(x_0)}|\nabla A|^{2}+ C_2R^{4-\frac {\delta} {\alpha}}    \end{align*}
for any $\rho<R\leq R_0$. Applying Lemma 2.1 of Chapter III in \cite {G} again, we obtain
\begin{align*}
 \int_{B_{\rho} (x_0)} |\nabla A|^{2}\,dV
  \leq C \rho^{4-\frac {\delta}{\alpha}} \quad \forall \rho \leq R_0.   \end{align*}
  The Morrey  growth lemma implies that $A$ is H\"older continuous at a neighborhood of the point $x_0\in M$. A standard procedure (e.g. \cite {G},  \cite {Isobe}, \cite {GH}) yields that $\nabla A$ is H\"older continuous and then smooth.

  \begin{acknowledgement}
 { The authors would like to thank Hao Yin for many useful discussions and suggestions.}
\end{acknowledgement}

\bibliographystyle{plain}

\end{document}